\documentclass[10pt]{amsart}

\usepackage{amsmath, amsthm}
\usepackage{eucal}

\usepackage{amssymb}
\usepackage{amscd}
\usepackage{palatino}
\usepackage{latexsym}
\usepackage{epsfig}
\usepackage{graphicx}
\usepackage{amsfonts}
\usepackage{psfrag}
\usepackage{caption} 
\usepackage{float}

\usepackage{url}
\usepackage{cite}

\usepackage[margin=1in]{geometry}

\input xy
\xyoption{all}
\UseComputerModernTips

\numberwithin{equation}{section}
\numberwithin{figure}{section}

\newtheorem{theorem}{Theorem}[section]
\newtheorem{lemma}[theorem]{Lemma}
\newtheorem{proposition}[theorem]{Proposition}

\newtheorem{remark}[theorem]{Remark}

\theoremstyle{definition}
\newtheorem{definition}[theorem]{Definition}
\newtheorem{example}[theorem]{Example}

\newcommand{\C}{{\mathbb{C}}}

\newcommand{\Z}{{\mathbb{Z}}}
\newcommand{\Q}{{\mathbb{Q}}}
\newcommand{\R}{{\mathbb{R}}}
\newcommand{\N}{{\mathbb{N}}}

\renewcommand{\P}{{\mathbb{P}}}

\renewcommand{\iff}{\Leftrightarrow}

\usepackage{bbold}

\DeclareMathOperator{\init}{in}

\newcommand{\A}{{\mathbb{A}}}
\newcommand{\val}{{\mathrm{val}}}
\newcommand{\trop}{{\mathrm{trop}}}

\usepackage{hyperref,tikz,subcaption}
\newcommand{\tree}{\tau}

\newcommand{\piod}{{\sf p}_{[1,d]}}
\newcommand{\pio}{{\sf p}_{[1]}}

\newcommand{\len}{{\mathrm{len}}}

\newcommand{\cone}{{\mathrm{Cone}}}

\newcommand\CO{{\mathcal O}}
\newcommand\D{{\mathcal D}}
\newcommand\mcP{{\mathcal P}}

\DeclareMathOperator{\Spec}{Spec}

\DeclareMathOperator{\conv}{conv}
\DeclareMathOperator{\ord}{ord}
\DeclareMathOperator{\Div}{div}

\begin{document}

\title{Wall-crossing for Newton-Okounkov bodies and the tropical Grassmannian}

\author{Laura Escobar}
\address{Department of Mathematics and Statistics\\ 
Washington University in St. Louis\\ One Brookings Drive \\ St. Louis, Missouri 63130 \\ U.S.A. }
\email{laurae@wustl.edu}
\urladdr{\url{http://www.math.wustl.edu/~lescobar/}}
\thanks{LE was partially supported by an Association for Women in Mathematics Mentoring Travel Grant, an AMS-Simons Travel Grant, and NSF Grant DMS 1855598}

\author{Megumi Harada}
\address{Department of Mathematics and
Statistics\\ McMaster University\\ 1280 Main Street West\\ Hamilton, Ontario L8S4K1\\ Canada}
\email{Megumi.Harada@math.mcmaster.ca}
\urladdr{\url{http://www.math.mcmaster.ca/Megumi.Harada/}}
\thanks{MH is partially supported by an NSERC Discovery Grant and a Canada Research Chair (Tier 2) award.}


\date{\today}

%%%%%%%%%%%%%%%%%%%%%
%  Abstract
%%%%%%%%%%%%%%%%%%%%%

\begin{abstract}
Tropical geometry and the theory of Newton-Okounkov bodies are two methods which produce toric degenerations of an irreducible complex projective variety. Kaveh-Manon showed that the two are related. We give geometric maps between the Newton-Okounkov bodies corresponding to two adjacent maximal-dimensional prime cones in the tropicalization of $X$. Under a technical condition, we produce a natural ``algebraic wall-crossing'' map on the underlying value semigroups (of the corresponding valuations). In the case of the tropical Grassmannian $Gr(2,m)$, we prove that the algebraic wall-crossing map is the restriction of a geometric map. In an Appendix by Nathan Ilten, he explains how the geometric wall-crossing phenomenon can also be derived from the perspective of complexity-one $T$-varieties; Ilten also explains the connection to the ``combinatorial mutations'' studied by Akhtar-Coates-Galkin-Kasprzyk.
\end{abstract}

\maketitle

\setcounter{tocdepth}{1}
\tableofcontents

\section{Introduction}\label{sec:intro}

Let $X$ be an irreducible complex
projective variety of dimension $d$. To study the geometry of $X$, we can study the central fiber of a \emph{toric degeneration} $\mathcal{X}$ of $X$, where a toric degeneration is a flat family of varieties whose central fiber $X_0$ is a toric variety; the fact that both $X$ and $X_0$ appear as fibers of the flat family $\mathcal{X}$ means that information about $X$ can be read off of $X_0$. The combinatorial data associated to toric varieties yield powerful tools for computing geometric invariants thereof. Hence, in the presence of a toric degeneration $\mathcal{X}$, it may be hoped that we can obtain geometric information about $X$ from the combinatorics associated to $X_0$.

In this paper, we focus on two well-known methods for constructing toric degenerations: tropical geometry, and the theory of Newton-Okounkov bodies. 
First we briefly recall the tropical geometry picture. Given a variety $X$ as above, realized as $\mathrm{Proj}(A) \cong \mathrm{Proj}(\C[x_1,\ldots,x_n]/I)$ where $A$ is its homogeneous coordinate ring and $\C[x_1,\ldots,x_n]/I$ a choice of presentation of $A$, the tropicalization  $\mathcal{T}(I)$ is a subset of $\R^n$ consisting of those (weight) vectors whose corresponding initial ideals $\init_w(I)$ contain no monomials (see \eqref{eq: def trop KM}).  
In fact, $\mathcal{T}(I)$ carries additional combinatorial structure, namely, it is a $(d+1$)-dimensional subfan of the Gr\"obner fan. A Gr\"obner degeneration of an ideal $I$ to the initial ideal $\init_w(I)$ yields a toric degeneration when the initial ideal $\init_w(I)$ is prime and binomial. Since primality is impossible if $\init_w(I)$ contains a monomial, the tropicalization $\mathcal{T}(I)$ can be viewed as the set of weight vectors which provide candidates for toric degenerations. 
Now we recall the point of view of Newton-Okounkov bodies. A valuation $\nu: A \setminus \{0\} \to \Q^{r}$ yields a multiplicative filtration on $A$ and hence an associated graded algebra $gr_{\nu}(A)$, whose grading is encoded in the value semigroup $S(A,\nu) := \mathrm{image}(\nu)$. When $\nu$ is full-rank, then (since $\C$ is algebraically closed) $\nu$ has one-dimensional leaves by Abhyankar's inequality (cf. \cite[Theorem 2.3]{KavehManon-published}, also \cite[Theorem 6.6.7]{HS06}), which implies that the associated graded $gr_{\nu}(A)$ is a semigroup algebra over the value semigroup $S(A,\nu)$. Hence, when $S(A,\nu)$ is finitely generated, $Proj$ of the associated graded is a (possibly non-normal) toric variety, and the associated degeneration of $A$ to $gr_{\nu}(A)$ is a toric degeneration \cite{Anderson}.

This manuscript was motivated by the results of Kaveh and Manon, who showed in \cite{KavehManon-published} that the two approaches sketched above are related. Let $C$ be a maximal-dimensional cone in $\mathcal{T}(I)$ and let $\init_C(I)$ denote the initial ideal associated to $C$. Assuming that this $\init_C(I)$ is prime, Kaveh and Manon show that the toric degeneration 
associated to $\init_C(I)$ can also be obtained from the point of view of Newton-Okounkov bodies. More precisely, they construct -- using a set of rational and linearly independent vectors $u_1,\ldots,u_{d+1}$ contained in the cone $C$ -- a valuation $\nu: A \setminus \{0\} \to \Q^{d+1}$ with respect to which the associated graded algebra $gr_{\nu}(A)$ of $A$ is isomorphic to the coordinate ring $\C[x_1,\ldots,x_n]/\init_C(I)$ obtained through Gr\"obner theory.

We can now sketch the first result of this paper. Suppose that $C_1$ and $C_2$ are both maximal-dimensional prime cones in $\mathcal{T}(I)$ and suppose they are adjacent, i.e., they share a codimension-$1$ face $C = C_1 \cap C_2$. First, we show that there are choices of $\{u_1, u_2,\ldots, u_{d+1}\} \in C_1$ and $\{u'_1, u'_2,\ldots, u'_{d+1}\} \in C_2$ such that the corresponding Newton-Okounkov polytopes $\Delta(A,\nu_{1})$ and $\Delta(A,\nu_{2})$ project to the same polytope under the linear projection $\piod: \R^{d+1} \to \R^d$ which forgets the last coordinate. We also show that the fibers are of the same Euclidean length (up to a global constant); we illustrate a very simple example in Figure~\ref{fig: fibers}.  The proof relies on variation of GIT quotients \cite{Hu}. Once we know that the fiber lengths are equal, it follows that there are two natural piecewise-linear maps $\mathsf{F}_{12}: \Delta(A, \nu_{1}) \to \Delta(A, \nu_{2})$ and $\mathsf{S}_{12}: \Delta(A, \nu_{1}) \to \Delta(A, \nu_{2})$, the ``flip'' and ``shift'' maps respectively, which behave as the identity on the first $d$ coordinates. We call these \textbf{(geometric) wall-crossing maps}. The precise statement is given in Theorem~\ref{theorem:main}. 

The geometric wall-crossing phenomenon for Newton-Okounkov bodies, as described above, can also be derived from the theory of complexity-one $T$-varieties. Specifically, the content of Theorem~\ref{theorem:main} can be obtained by adapting the arguments in 
\cite{Pet}, which describe Newton-Okounkov bodies for normal complexity-one T-varieties. (More details are in the Appendix.) 
This was observed by Ilten and Manon already in 2017 although not recorded explicitly in \cite{IM}. In the Appendix by Nathan Ilten, this complexity-one perspective is briefly explained; in addition, Ilten explains the connection to the ``combinatorial mutations'' of polytopes, as studied by Akhtar, Coates, Galkin, and Kasprzyk \cite{ACGK}. 

We now describe the second set of results in this paper. In addition to the ``geometric'' wall-crossing maps discussed above, under a certain technical hypothesis (stated precisely in Section~\ref{subsec: algebraic crossing}), it is also possible to construct -- using a set of standard monomials coming from Gr\"obner theory -- a natural bijection $\Theta: S(A, \nu_{1}) \to S(A, \nu_{2})$ commuting with the projection $\piod$. We call this the \textbf{algebraic} wall-crossing. In general, the map $\Theta$ is not straightforward to compute. Since the semigroups $S(A,\nu_{i})$ for $i=1,2$ are subsets of the respective cones $P(A,\nu_{i}) := \mathrm{cone}(\Delta(A,\nu_{i}))$ and the maps $\mathsf{F}_{12}$ and $\mathsf{S}_{12}$ naturally extend to the level of the cones, it is natural to ask whether $\Theta$ is simply the restriction to $S(A,\nu_{1})$ of either of the geometric wall-crossing maps. In Example~\ref{example: algebraic is not geometric} we show that, in general, the answer is no.  However, for the case of the tropical Grassmannian of $2$-planes in $m$-space, we show that the algebraic  wall-crossing map $\Theta$ is the restriction of the ``flip'' map $\mathsf{F}_{12}$; this is recorded in Theorem~\ref{theorem: main Gr2m}.

The results of this paper suggest some natural directions for future work; we mention a small sample. First, 
our Theorem~\ref{theorem: main Gr2m} motivates the natural question: under what conditions is the algebraic wall-crossing map a restriction of a geometric wall-crossing? Secondly, and as a special case,
it seems natural to ask whether our analysis of the algebraic and geometric wall-crossing for $Gr(2,m)$ can be generalized to the tropicalizations of the higher Grassmannians $Gr(k,m)$ for $k>2$. Recent work of Mohammadi and Shaw \cite{MoSh} on $\trop(Gr(3,m))$ suggest that the case $k=3$ may be tractable.
In addition, it is well-known that the Grassmannian $Gr(2,m)$ is a cluster variety, and in this special case, our algebraic wall-crossing $\Theta$ can be seen to be related to cluster mutation. In light of the work
of Rietsch and Williams (e.g. \cite[Corollary 11.16]{Rietsch-Williams}) we hope to better understand, in more generality, the connections between (both the geometric and algebraic) wall-crossing maps and clusters.

We now briefly outline the layout of this paper. In Section~\ref{section:background} we establish the notation and setup for the rest of the paper. In particular, we state precisely the result of Kaveh and Manon, on which this paper relies. We then give a statement of our first main result in Theorem~\ref{theorem:main}, namely, that the fiber lengths are equal. In Section~\ref{sec:fibers} we give a proof of half of Theorem~\ref{theorem:main}, which we formalize in Theorem~\ref{theorem: fiber lengths equal}. In Section~\ref{sec:wall-crossing} we prove the second half of Theorem~\ref{theorem:main}, namely, we construct the geometric  ``shift'' and ``flip'' wall-crossing maps; once we know the equality of fiber lengths, this is quite straightforward. Moreover, in Section~\ref{subsec: algebraic crossing} we define, under an additional technical hypothesis, an ``algebraic wall-crossing'' on the semigroups associated to $C_1$ and $C_2$. We also show that, in general, the algebraic wall-crossing need not arise from either of the geometric wall-crossing maps. Section~\ref{sec:Gr2m} is devoted to the tropical Grassmannian of $2$-planes in $\C^m$, and we work out in detail what our results entail for this special case, including a concrete formula for the ``flip'' geometric wall-crossing map in this case. We prove our main result of this section -- that in this case, the algebraic wall-crossing is the restriction of the ``flip'' geometric map -- in Section~\ref{subsec: algebraic is flip for Gr2m}. Finally, the Appendix by Nathan Ilten discusses the complexity-one $T$-variety perspective.

%%%%%%%%%%%%%%%%%%%%%
%  Section - Background
%%%%%%%%%%%%%%%%%%%%%

\section{Background: Newton-Okounkov bodies and tropical geometry}
\label{section:background}

In this section we briefly recall the background necessary for the statement of our main theorem (Theorem~\ref{theorem:main}). Throughout, $X$ is an irreducible complex projective variety of dimension $d$ and $A$ denotes its homogeneous coordinate ring.  In particular, $A$ is a finitely generated $\C$-algebra and is positively graded. Moreover, from the assumptions on $X$ it follows that $A$ is a domain and has Krull dimension $d+1$. 

We begin with a brief account of the theory of Newton-Okounkov bodies; see \cite{KavehManon-published} for details. We restrict to the setting above. Let $r$ be an integer, $0 < r \leq d$.
Let $\prec$ denote a total order on $\Q^r$ which respects addition.

\begin{definition}\label{def:valuation} (\cite[Definition 2.1]{KavehManon-published}) 
Consider $(\Q^r, \prec)$ as an abelian group equipped with the total order $\prec$. A function $\nu: A \setminus \{0\} \to \Q^r$ is a \textbf{valuation} over $\C$ if 
\begin{enumerate} 
\item for all $0 \neq f,g$ in $A$ with $0 \neq f+g$ we have $\nu(f+g) \succeq \min \{\nu(f), \nu(g)\}$
\item for all $0 \neq f, g$ in $A$ we have $\nu(fg) = \nu(f)+\nu(g)$ and 
\item for all $0 \neq f$ and $0 \neq c \in \C$ we have $\nu(cf) =\nu(f)$, or equivalently, $\nu(c)=0$ for all $0 \neq c \in \C$. \qedhere
\end{enumerate} 
\end{definition}

The valuation $\nu$ also gives rise to a multiplicative filtration $\mathcal{F}_\nu$ on $A$ as follows. For $a \in \Q^r$ we define  
	\begin{equation}\label{eq: multiplicative filtration for valuation}
	F_{\nu \succeq a} := \{ f \in A \setminus \{0\}  \, \vert \, \nu(f) \succeq a \} \cup \{0\} 
	\qquad
	\text{and}
	\qquad
	F_{\nu \succ a} := \{ f \in A \setminus \{0\}  \, \vert \, \nu(f) \succ a \} \cup \{0\}.
	\end{equation}
A valuation $\nu:A\setminus\{0\}\to\Q^r$ has \textbf{one-dimensional leaves} if for every $a\in\Q^r$ the vector space $ F_{\nu \succeq a} / F_{\nu \succ a}$ is at most one-dimensional.
The \textbf{associated graded algebra $gr_{\nu}(A)$} is defined to be 
	\begin{equation}\label{eq: associated graded for valuation} 
	gr_{\nu}(A)  = \bigoplus_{a \in \Q^r} F_{\nu \succeq a} / F_{\nu \succ a}. 
	\end{equation}
	The ring structure on $gr_{\nu}(A)$ is induced from the ring structure on $A$. 
By construction, $gr_{\nu}(A)$ is graded by $S(A,\nu)$ since $F_{\nu \succeq a} / F_{\nu \succ a} \neq 0$ if and only if $a\in S(A,\nu)$. 
Note that an element $g \in A \setminus \{0\}$ can be mapped to the associated graded $gr_{\nu}(A)$ by considering its associated equivalence class in the quotient $F_{\nu \succeq a}/F_{\nu \succ a}$, where $a=\nu(g)$. Also, having one-dimensional leaves implies that, given a vector space basis for $gr_\nu(A)$ which is homogeneous with respect to its grading, the map which sends an element of the basis to its degree is a bijection.  

We restrict attention to valuations of the following form. For a positively graded algebra $A = \oplus_{k \geq 0} A_k$, we say that a valuation $\nu$ is \textbf{homogeneous} on $A$ if the following holds: for any $0 \neq f_1 \in A$ and $0 \neq f_2 \in A$, if $\deg(f_1) < \deg(f_2)$ then $\nu(f_1) \succ \nu(f_2)$ (note the switch).  
Specifically, we always assume we have a valuation $\nu: A \setminus \{0\} \to \N \times \Q^{r-1} \subseteq \Q^r$ such that its first component is the degree, i.e. 
\begin{equation}\label{eq:valuation on pos graded}
\nu(f) = (\deg(f), \nu)  : A \setminus \{0\} \to \N \times \Q^{r-1}.
\end{equation}
where the total order on $\N \times \Q^{r-1}$ is defined as follows: for $(a,v), (b,w) \in \N \times \Q^{r-1}$, 
\begin{equation}\label{eq:total order with degree}
(a, v) \preceq (b,w) \textup{ if and only if }  ( a > b, \textup{ or, } (a=b \textup{ and } v \preceq_{\Q^{r-1}} w))
\end{equation}
where the order $\preceq_{\Q^{r-1}}$ on $\Q^{r-1}$ is taken to be the standard lex order.  Note this ordering first compares the first coordinates and then breaks ties with the remaining coordinates; moreover, there is a reversal of the ordering on the first coordinate. Clearly, such a valuation is homogeneous.  
We additionally assume $\nu: A \setminus \{0\} \to \Q^r$ is a discrete \footnote{A valuation is \textbf{discrete} if the image of the valuation is discrete in the target (in other words, for any $y \in \nu(A \setminus \{0\})$, there exists an open neighborhood $U$ of $y$ in $\Q^r$ such that $U \cap \nu(A \setminus \{0\}) = \{y\}$).} valuation. 
The image $S(A,\nu) := \nu(A \setminus \{0\}) \subseteq \Q^r$ of such a valuation is a discrete additive semigroup of $\Q^r$ and is called the \textbf{the value semigroup (of $\nu$)}.
The \textbf{rank of the valuation} is the rank of the group generated by its value semigroup.

\begin{definition}\label{definition: NOBY} 
Let $A$ be the homogeneous coordinate ring of a projective variety and $\nu$ a discrete homogeneous valuation on $A$.
 The \textbf{Newton-Okounkov cone} of $(A,\nu)$ is the convex set 
 \[
\mathrm{Cone}(S(A,\nu)) := \left\{ \sum_{i=1}^n t_i s_i \, \mid \, s_i \in S(A,\nu), t_i \in \R_{\geq 0} \right\} \subseteq \R^{d+1},
\]
 i.e. the non-negative real span of elements of $S(A,\nu)$. 
 The \textbf{Newton-Okounkov body} of $(A,\nu)$ is the convex set $\Delta(A,\nu):=\{x_1=1\}\cap \cone(S(A,\nu))$.
\end{definition}

Following \cite{KavehManon-published} we say that a set $\mathcal{B} \subseteq A \setminus \{0\}$ is a \textbf{Khovanskii basis} for $(A,\nu)$ if the image of $\mathcal{B}$ in $gr_{\nu}(A)$ forms a set of algebra generators of $gr_{\nu}(A)$. 
Note that existence of a finite Khovanskii basis for $(A,\nu)$ implies that the associated value semigroup $S(A,\nu)$ is finitely generated, which in turn means that the Newton-Okounkov body of Definition~\ref{definition: NOBY} is a convex rational polytope, and thus is a combinatorial object.

We next briefly recall some basic terminology in tropical geometry; for details see \cite{KavehManon-published}. Let $A$ be an algebra as above and suppose $\mathcal{B} = \{b_1, \ldots, b_n\}$ is any finite set of algebra generators of $A$ which we assume to be homogeneous of degree $1$. Consider the surjective $\C$-algebra homomorphism
\begin{equation}\label{eq:surjection to A}
\pi: \C[x_1, \ldots, x_n] \to A
\end{equation}
defined by $\pi(x_i) = b_i$ for $1 \leq i \leq n$. This is a map of graded rings provided that we define the grading on the polynomial ring by $\deg(x_i)=1$ for all $i$. Let $I := \ker(\pi) \subseteq \C[x_1,\ldots,x_n]$
which is homogeneous since $\pi$ preserves degrees. Then we have a natural presentation $A \cong \C[x_1, \ldots, x_n]/I$
associated to this choice of generating set $\mathcal{B}$, realizing
$Spec(A)$ explicitly as a subvariety of $\C^n$. Note that $Spec(A)$ is the 
affine cone over $X \cong Proj(A)$ so we use the notation $\tilde{X} := Spec(A).$

As noted in \cite[Introduction]{KavehManon-published}, conceptually it is more appropriate to talk about the tropicalization of a subvariety of a torus. Geometrically, this corresponds to looking at the 
intersection $\tilde{X}^0 := \tilde{X} \cap (\C^*)^n \subseteq \mathbb{A}^n$ of $\tilde{X}$ with the torus $(\C^*)^n$ sitting naturally in $\mathbb{A}^n$. Algebraically, this corresponds to looking at the algebra 
\begin{equation}\label{eq: torus part} 
\C[x_1^{\pm 1}, \cdots, x_n^{\pm 1}]/ I^L, \quad \textup{ where } I^L :=  I \cdot \C[x_1^{\pm 1}, \cdots, x_n^{\pm 1}].
\end{equation}
We will consider both $I$ and $I^L$ below. 
Following \cite{KavehManon-published} we define the \textbf{tropicalization} $\mathcal{T}(I)$ (or \textbf{tropical variety}) of $\tilde{X}$ corresponding to the choice of presentation $A \cong \C[x_1,\ldots,x_n]/I$ by 
\begin{equation}\label{eq: def trop KM}
\mathcal{T}(I) := \{ w \in \Q^n \, \vert \, \init_w(I) \textup{ does not contain any monomials } \}.
\end{equation} 
This definition is a priori different from the definition appearing in \cite[Section 3.2]{MaclaganSturmfels}, 
but it is not difficult to see that they are in fact equivalent. 
Thus 
$\mathcal{T}(I)$ is a polyhedral fan which is pure of dimension $d+1$ and also 
is a subfan of the Gr\"obner fan \cite[Proposition 3.2.8, Theorem 3.3.5]{MaclaganSturmfels}. For each cone $C$ in $\mathcal{T}(U)$ there is a unique initial ideal denoted by $\init_C(I)$ associated to this cone, defined to be $\init_{\omega}(I)$ for any $\omega$ in the interior of $C$.

\begin{definition}\label{def: prime cone} 
We say a cone $C$ in $\mathcal{T}(I) = \trop(X^0)$ is a \textbf{prime cone} if the corresponding initial ideal $\init_C(I)$ is a prime ideal.
A \textbf{maximal-dimensional prime cone} is a prime cone $C$ with maximal dimension, i.e., $\dim_{\R}(C)=d+1$.  
(In \cite{KavehManon-published} they use the terminology ``maximal prime cone'' instead.) 
\end{definition}

To state the result of Kaveh and Manon which relates Newton-Okounkov theory to tropicalizations, we need the notion of a quasivaluation, which is nearly identical to that of a valuation (cf. Definition~\ref{def:valuation}) except that we allow for superadditivity in the multiplication. 

\begin{definition}\label{def: quasivaluation} (\cite[Definition 2.26]{KavehManon-published}) 
Consider $(\Q^r, \prec)$ as an abelian group equipped with the total order $\prec$. 
Let $A$ be a $\C$-algebra. A function $\nu: A \setminus \{0\} \to \Q^r \cup \{\infty\}$ is a \textbf{quasivaluation} over $k$ if 
\begin{enumerate} 
\item For all $0 \neq f, g, f+g$ we have $\nu(f+g) \succeq \min\{\nu(f), \nu(g)\}$.
\item For all $0 \neq f, g \in A$ we have $\nu(fg) \succeq \nu(f) + \nu(g)$. 
\item For all $0 \neq f \in A$ and $0 \neq c \in \C$ we have $\nu(cf)=\nu(f)$. \qedhere
\end{enumerate} 
\end{definition} 
As in the case of valuations, a quasivaluation gives rise to a filtration of the original algebra, as well as an associated graded algebra, by using the same formulas~\eqref{eq: multiplicative filtration for valuation} and~\eqref{eq: associated graded for valuation}. 
Conversely, one can construct a quasivaluation from a decreasing algebra filtration $\mathcal{F} = \{F_a\}_{a \in \Q^r}$ of $A$ by $\C$-subspaces by defining, for any $0 \neq f \in A$, 
\begin{equation}\label{eq: def quasival from filtration} 
\nu_{\mathcal{F}}(f) := \max \{ a\in \Q^r : f \in F_a\}.
\end{equation}
(If the max is not attained, we define $\nu_{\mathcal{F}}(f) := \infty$.) The quasivaluations which are central to Kaveh and Manon (and also for this paper) all arise in this manner via a pushforward filtration, as we now describe.

Let $M \in \Q^{r\times n}$ be a matrix. For $p = \sum_{\alpha} c_{\alpha} x^{\alpha} \in \C[x_1, \ldots, x_n]$ we define the $\Q^r$-valued \textbf{ weight valuation} $\tilde{\nu}_M: \C[x_1,\ldots,x_n]  \setminus \{0\} \to \Q^r$ associated to $M$ by
	\begin{equation}\label{eq: weight valuation of matrix}
	\tilde{\nu}_M(p) := \min \{ M \alpha : c_{\alpha} \neq 0 \} \in \Q^r
	\end{equation}
where $M \alpha$ is the usual matrix multiplication, the exponent vector $\alpha$ is treated as a column vector, and the minimum is taken with respect to the fixed total ordering on $\Q^r$. \footnote{Note that if $M \in \Q^{1 \times n}$ is a single row vector, then $M \alpha$ is just the usual inner product pairing of a ``rank-$1$ weight vector'' against the exponent vector $\alpha$, and the above rule recovers the usual Gr\"obner theory.}
In Gr\"obner theory one frequently takes a maximum, but in this paper we take the ``minimum'' convention. Similarly we define the initial form of $p$ with respect to $M$ by 
\[
\init_M(p) = \sum_{\beta: M\beta = \tilde{\nu}_M(p)} c_\beta x^\beta
\]
so we take only those terms with minimal value of $M\beta$. We define the initial ideal of $I$ with respect to $M$, denoted $\init_M(I)$, to be the ideal generated by all $\init_M(p)$ for $p \in I$. Note that $\init_M(I)$ is $M$-homogeneous in the sense that $h = \init_M(h)$ for all $h \in \init_M(I)$. 
Next, let $\tilde{\mathcal{F}}_M$ denote the (decreasing) filtration on $\C[x_1,\ldots,x_n]$ obtained from $\tilde{\nu}_M$. 
We define 
the \textbf{weight filtration} of $A$ (associated to the surjection $\pi$ and the matrix $M$) to be the pushforward filtration $\pi(\tilde{\mathcal{F}}_M)$ on $A$ given by the surjection in~\eqref{eq:surjection to A}. 
The \textbf{weight quasivaluation} $\nu_M$ is the quasivaluation on $A$ associated to this weight filtration as defined by~\eqref{eq: def quasival from filtration}. In general, it need not be a valuation. 

We say that a $\mathbb{C}$-vector space basis $\mathbb{B}$ for $A$ is an \textbf{adapted basis} for $(A,\nu)$ if the image of $\mathbb{B}$ in $gr_{\nu}(A)$ forms a vector space basis for $gr_{\nu}(A)$. 
In the case of a valuation of the form $\nu_M$ for some $M$ as above, we will see below (Theorem~\ref{theorem:kavehmanon theorem 4}) that an adapted basis can be obtained through Gr\"obner theory. Recall that the maximal cones of the Gr\"obner fan are indexed by monomial orders $<$;
let $C_<$ denote the maximal cone corresponding to $<$. Now suppose $C$ is a cone of the tropicalization $\mathcal{T}(I)$ which is also a face of $C_<$. Let $\mathcal{S}(<,I) \subseteq \C[x_1,\ldots,x_n]$ denote the set of \textbf{standard monomials with respect to $I$ and $<$}, i.e. the monomials not contained in $\init_<(I)$.
It is well-known that the projection onto $ \C[x_1,\ldots,x_n]	\big/	I$ of the monomials $\mathcal{S}(<,I)$ form a vector space basis for $ \C[x_1,\ldots,x_n]	\big/	I$.  Finally, recall that the \textbf{Gr\"obner region} $GR(I) \subseteq \R^n$ of an ideal $I$ is the set of $u \in \Q^n$ such that there exists a monomial order $<$ such that $\init_<(\init_u(I)) = \init_<(I)$.  
We have the following theorem, which motivates the current manuscript.

\begin{theorem}\label{theorem:kavehmanon theorem 4} (\cite[Propositions 4.2 and 4.8]{KavehManon-published}) 
Following the notation in this section, let $C \subset \mathcal{T}(I)$ be a maximal-dimensional prime cone. Let $\{u_1, \ldots, u_{d+1}\} \subset C$ be a collection of rational vectors which span a real vector space of dimension $d+1 = dim(C)$.  Let $M \in \mathrm{Mat}((d+1) \times n, \Q)$ be the $(d+1) \times n$ matrix whose row vectors are $u_1, \ldots, u_{d+1}$. Let $\nu_M: A \setminus \{0\} \to \Q^{d+1}$ denote the corresponding weight quasivaluation. Then $\nu_M$ is a valuation, and the following hold: 

\begin{enumerate} 
\item $gr_{\nu_M}(A) \cong \C[x_1, \ldots, x_n]/ \init_M(I)$ as $\Q^{d+1}$-graded algebras, 
\item  If $C$ lies in the Gr\"obner region of $I$, the valuation $\nu_M$ has an adapted basis which can be taken to be the projection via $\pi$ of the standard monomial basis $\mathcal{S}(<,I)$ for a maximal cone $C_<$ in the Gr\"obner fan of $I$ containing $C$. \qedhere
\end{enumerate} 
\end{theorem} 

\begin{remark} 
By an argument similar to \cite[Proposition 1.12]{Sturmfels}, if $I$ is a homogeneous ideal then its Gr\"obner region equals $\Q^n$, so in our case, the hypothesis in item (2) above always holds. 
\end{remark}

From item (1) of Theorem~\ref{theorem:kavehmanon theorem 4} it follows from basic tropical theory that 
the value semigroup $S(A,\nu_M)$ (which is the semigroup of the toric variety corresponding to $\C[x_1,\ldots,x_n]/\init_C(I)$) is generated by the column vectors of the matrix $M$, and also that the associated Newton-Okounkov body $\Delta(A,\nu_M)$ can be explicitly computed as
\begin{equation}\label{eq: NOBY as convex hull} 
\Delta(A, \nu_M) = \textup{ convex hull of the columns of } M.
\end{equation}

The results above suggest that there should be a straightforward relationship between the Newton-Okounkov bodies associated to two maximal-dimensional prime cones $C_1$ and $C_2$ in $\mathcal{T}(I)$ if they are \emph{adjacent} in $\mathcal{T}(I)$, i.e., they share a codimension-$1$ face $C := C_1 \cap C_2$. 
The goal of this manuscript is to describe such a ``wall-crossing phenomenon'' for Newton-Okounkov bodies and to work out the case of the Grassmannians $Gr(2,m)$. The first main result is Theorem~\ref{theorem:main} below. 
To state the theorem, we need some preparation. 
For $C, C_1$ and $C_2$ as above, fix, once and for all, a linearly independent set $\{u_1, u_2, \ldots, u_d\}$ of integral vectors contained in $C$. In particular, $\{u_1, \ldots, u_d\}$ span a real vector space of dimension $d=\dim_\R(C)$.  
We also fix a total order $\prec$ satisfying \eqref{eq:total order with degree}.
We may assume that $u_1$ is chosen to be the vector $(1,1,\ldots,1)$ (this is possible because the ideal $I$ is homogeneous); this ensures that the corresponding weight valuation is homogeneous. We also fix integral vectors $w_1 \in C_1$ and $w_2 \in C_2$ such that $w_1 + \sum_j u_j$ (respectively $w_2 + \sum_j u_j$) lies in the interior of $C_1$ (respectively $C_2$). 
Let $M$ be the $d \times n$ matrix whose $j$-th row is the vector $u_j$ chosen above, and let $M_1$ (respectively $M_2$) denote the $(d +1) \times n$ matrix whose top $d$ rows are the same as those in $M$ and whose bottom $(d+1)$-st row is equal to $w_1$ (respectively $w_2$).

Let $\nu_{M_1}, \nu_{M_2}$ and $\nu_M$ be the corresponding weight quasivaluations on $A$. 
Theorem~\ref{theorem:kavehmanon theorem 4} implies that $\nu_{M_1}, \nu_{M_2}$ are valuations.
Although we remarked above that $\nu_M$ for arbitrary $M$ need not be a valuation, for $M$ chosen as in our setting, we will prove in Lemma~\ref{lemma: projection of valuation} that $\nu_M=\piod\circ\nu_{M_i}$ for $i=1,2$. 
It can be deduced that $\nu_M$ is also a valuation from the fact that $\nu_{M_i}$ are valuations and $\piod$ is a linear projection to the first $d$ coordinates.

\begin{theorem}\label{theorem:main} 
Let $A = \oplus_k A_k$ be a positively graded algebra over $\C$, and assume $A$ is an integral domain and has Krull dimension $d+1$.
 Let $\mathcal{B} = \{b_1, \ldots, b_n\}$ be a subset of $A_1$ (the homogeneous degree $1$ elements of $A$) which generate $A$ as an algebra. 
Let $I$ be the homogeneous ideal such that the presentation induced by $\mathcal{B}$ is $A \cong \C[x_1, \ldots, x_n]/I$ (as in \eqref{eq:surjection to A}), and let $\mathcal{T}(I)$ denote its tropicalization. 
Suppose that $C_1$ and $C_2$ are two maximal-dimensional prime cones in $\mathcal{T}(I)$ 
that share a codimension-1 face $C$.
Let $M_1$, $M_2$, and $M$ be the matrices described above and $\nu_{M_1}, \nu_{M_2}$ and $\nu_M$ the corresponding weight valuations on $A$. 
Let $\Delta(A, \nu_{M_1}) \subseteq \{1\} \times \R^d$, $\Delta(A, \nu_{M_2}) \subseteq \{1\} \times \R^d$ and $\Delta(A, \nu_{M}) \subseteq \{1\} \times \R^{d-1}$ denote the corresponding Newton-Okounkov bodies. Let $\piod: \R^{d+1} \to \R^d$ denote the linear projection $\R^{d+1} \to \R^d$ obtained by deleting the last coordinate. 
Then
	\[
	\piod(\Delta(A,\nu_{M_1})) = \piod(\Delta(A,\nu_{M_2})) = \Delta(A, \nu_M)
	\]
	and for any $\xi \in \Delta(A,\nu_M)$, the Euclidean lengths of the fibers $\piod^{-1}(\xi) \cap \Delta(A,\nu_{M_1})$ and $\piod^{-1}(\xi) \cap \Delta(A,\nu_{M_2})$ are equal, up to a global constant which is independent of $\xi$. 
Moreover, there exist two piece-wise linear identifications $\mathsf{S}_{12}: \Delta_{M_1} \to \Delta_{M_2}$ and $\mathsf{F}_{12}: \Delta_{M_1} \to \Delta_{M_2}$, called the ``shift map'' and the ``flip map'' respectively, which have the following properties: for $\Phi_{12} \in \{\mathsf{S}_{12}, \mathsf{F}_{12}\}$, we have that 
the diagram 
\[
\xymatrix{
\Delta(A,\nu_{M_1}) \ar[rd]_{\piod} \ar[rr]^{\Phi_{12}} && \Delta(A,\nu_{M_2}) \ar[ld]^{\piod} \\
 & \Delta(A,\nu_M)
}
\]
commutes, and $\Phi_{12}$ preserves the Euclidean lengths of the fibers of $\piod$. 
\end{theorem} 

\begin{remark}\label{remark: the global constant} 
The global constant appearing in Theorem~\ref{theorem:main} above depends only on the choices of the matrices $M_1, M_2$ and $M$ which represent the cones $C_1, C_2$ and $C$ respectively, which is why the constant is independent of the choice of basepoint $\xi \in \Delta(A,\nu_M)$. 
\end{remark} 

\begin{example}\label{ex: main example}
We illustrate Theorem~\ref{theorem:main} in an example which is explained in detail in Section~\ref{sec:wall-crossing}. In this example we can see explicitly that the lengths of the fibers under ${\sf p}_1$ and ${\sf p}_2$ are the same length; see Figure~\ref{fig: fibers}. 
	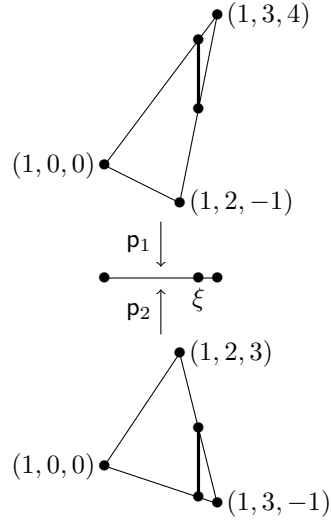
\begin{figure}[h]
	\begin{tikzpicture}\begin{scope}[scale=.5]
	\draw (0,0) node {$\bullet$} node[left] {$(1,0,0)$}--(2,-1) node {$\bullet$} node[right] {$(1,2,-1)$}--(3,4) node {$\bullet$} node[right] {$(1,3,4)$}--(0,0);
	%\draw (1,0) node {$\bullet$};
	\draw[->] (1.5,-1.5) --node[left] {${\sf p}_1$}(1.5,-2.7);
	\draw (0,-3) node {$\bullet$}--(3,-3) node {$\bullet$};
	\draw[<-] (1.5,-3.3) --node[left] {${\sf p}_2$}(1.5,-4.5);
	\draw (0,-8) node {$\bullet$} node[left] {$(1,0,0)$}--(2,3-8) node {$\bullet$} node[right] {$(1,2,3)$}--(3,-1-8) node {$\bullet$} node[right] {$(1,3,-1)$}--(0,-8);
	\draw (2.5,-3) node {$\bullet$} node[below] {$\xi$};
	\draw [very thick] (2.5,3.33333333333) node {$\bullet$}--(2.5,1.5) node{$\bullet$};
	\draw [very thick] (2.5,-8-0.83333333333) node {$\bullet$}--(2.5,-8+1) node{$\bullet$};
	%\draw (1,0) node {$\bullet$};
	\end{scope}
	\end{tikzpicture}
	\caption{Two polytopes projecting onto a common interval, and their fibers under the projection maps.}\label{fig: fibers}
	\end{figure}
\end{example}

The next two sections are devoted to a proof of Theorem~\ref{theorem:main}.

%%%%%%%%%%%%%%%%%%%%%
%  Section - Fiber lengths
%%%%%%%%%%%%%%%%%%%%%

\section{The fiber lengths are equal}\label{sec:fibers}

The purpose of this section is to prove the first half of Theorem~\ref{theorem:main}. Specifically, we show in Lemma~\ref{lemma: NOBY projects to NOBY} that we have a diagram 
\begin{equation}\label{eq: projection diagram}
\xymatrix{
\Delta(A,\nu_{M_1}) \ar[rd]_{\piod}  && \Delta(A,\nu_{M_2}) \ar[ld]^{\piod} \\
 & \Delta(A,\nu_M)
}
\end{equation}
relating the $3$ polytopes;
then in Theorem~\ref{theorem: fiber lengths equal} we show the second assertion of Theorem~\ref{theorem:main}, namely, that the fiber lengths are equal (up to a global constant -- cf. Remark~\ref{remark: the global constant}). Theorem~\ref{theorem: fiber lengths equal} is the substantive assertion of Theorem~\ref{theorem:main}, and our argument uses a variation of GIT quotients. 
In addition to the projection $\piod$, we will also use $\pio:\R^{d+1}\to\R$, the projection which maps onto the first coordinate with respect to the standard basis.

\begin{lemma}\label{lemma: NOBY projects to NOBY}
Following the notation in this section, the images under the projection $\piod: \R^{d+1} \to \R^{d}$ of $\Delta(A,\nu_{M_1})$ and $\Delta(A,\nu_{M_2})$ are the same and are equal to the Newton-Okounkov body associated to $\nu_M$, i.e. 
\[
\piod(\Delta(A,\nu_{M_1})) = \piod(\Delta(A,\nu_{M_2})) = \Delta(A,\nu_M).
\]
\end{lemma}

For the proof of Lemma~\ref{lemma: NOBY projects to NOBY} the following is useful (see \cite[Lemma 3.2]{KavehManon-published} and remarks following). In analogy to the classical Gr\"obner theory, we say the \textbf{(rank $r$) Gr\"obner region} $GR^r(I) \subseteq \R^{r \times n}$ is the set of $M$ such that there exists a monomial order $<$ with $\init_<(\init_M(I)) = \init_<(I)$.

\begin{lemma}\label{lemma: definition nuMf}
Following the above notation, for any $f \in \C[x_1,\ldots,x_n]/I$ we have  
\begin{equation}\label{eq: definition nuMf}
\nu_M(f) = \max \{ \tilde{\nu}_M(\tilde{f}) \, \mid \, \tilde{f} \in \C[x_1,\ldots,x_n] \textup{ and } \pi(\tilde{f}) = f \}.
\end{equation}
Moreover, in our setting, maximum on the RHS of the above equation is always attained. 
\end{lemma} 

\begin{proof} 
The first claim is \cite[Lemma 3.2]{KavehManon-published}.
The second claim follows from \cite[Proposition 3.3, Lemma 8.7]{KavehManon-published} and the remarks following \cite[Definition 2.27]{KavehManon-published}.
\end{proof}

Using the above, we can explicitly compute $\nu_M$ as follows.

\begin{lemma} \label{lemma: projection of valuation}
Following the notation above,
$
\piod \circ \nu_{M_1} = \piod \circ \nu_{M_2} = \nu_M.
$
\end{lemma}

\begin{proof} 
The valuation $\tilde{\nu}_M$ is by definition a minimum, i.e. 
$\tilde{\nu}_M(\tilde{f}) = \mathrm{min} \{ M \alpha \, \mid \, c_\alpha \neq 0 \}$
for $\tilde{f} = \sum_\alpha c_\alpha x^\alpha$ and similarly for $\tilde{\nu}_{M_1}$ and $\tilde{\nu}_{M_2}$. Therefore, the formula~\eqref{eq: definition nuMf} is 
in fact a max-min formula. 
Moreover, by our assumption on $M_1, M_2$ and $M$ we know that $\piod(M_1 \alpha) = \piod(M_2 \alpha) = M \alpha$. 
To prove the lemma we first prove that 
\begin{equation}\label{eq: piod commutes with min and max} 
\piod(\min T) = \min \piod(T) \quad \textup{ and } \quad  \piod(\max T) = \max \piod(T)
\end{equation}
for any $T\subset \Z^{d+1}_{\geq 0}$ such that both $\min T$ and $\max T$ exist. Indeed, from the definition of the total order~\eqref{eq:total order with degree} we have that 
$a \preceq b$ implies $\piod(a) \preceq \piod(b)$
for any $a,b \in \Z^{d+1}_{\geq 0}$. Then it readily follows that if $T$ achieves its min (respectively max) then the left (respectively right) equation of~\eqref{eq: piod commutes with min and max} holds. 
Now suppose $\tilde{f} = \sum_{\alpha} c_{\alpha} x^{\alpha}\in\C[x_1,\ldots,x_n]$. For any $N \in \{M, M_1, M_2\}$ we define $T_{\tilde{f},N} := \{ N\alpha \, \mid \, c_\alpha \neq 0\}$. Since $T_{\tilde{f},N}$ is finite, it achieves both its minimum and maximum, and by definition $\tilde{\nu}_N(\tilde{f}):= \min T_{\tilde{f}, N}$. For $f \in \C[x_1,\ldots,x_n]/I$ and $i=1,2$, we define 
	\[
	T_{f,M_i}=\{ \min T_{\tilde{f},M_i} \, \mid \, \tilde{f} \in \C[x_1,\ldots,x_n] \textup{ and } \pi(\tilde{f}) = f \}.
	\]
By the last claim of Lemma~\ref{lemma: definition nuMf} we know that the maximum of $T_{f,M_i}$ is achieved for $i=1,2$, and therefore $\piod(\max T_{f,N}) = \max \piod(T_{f,N})$. As observed above, $T_{\tilde{f},N}$ also achieves its minimum, so that $\piod(\min T_{\tilde{f},n}) = \min \piod(T_{\tilde{f},N})$.
From the above we can compute that for $i=1,2$
	\begin{align*}
	\piod ( \nu_{M_i} (f))
	&=
	\piod (\max T_{f,M_i}) \quad \textup{ by~\eqref{eq: definition nuMf}} 
	\\&=
	\max \piod (T_{f,M_i}) \quad \textup{ since the max of $T_{f, M_i}$ is achieved} 
	\\&=
	\max \{ \piod (\min T_{\tilde{f},M_i})\, \mid \, \tilde{f} \in \C[x_1,\ldots,x_n] \textup{ and } \pi(\tilde{f}) = f   \}
	\quad \textup{ by definition of $T_{f,M_i}$} 
	\\&=
	\max \{ \min \piod (T_{\tilde{f},M_i})\, \mid \, \tilde{f} \in \C[x_1,\ldots,x_n] \textup{ and } \pi(\tilde{f}) = f   \}
	\quad \textup{ since $T_{\tilde{f}, M_i}$ is finite } 
	\\&=
	\max \{ \min T_{\tilde{f},M}\, \mid \, \tilde{f} \in \C[x_1,\ldots,x_n] \textup{ and } \pi(\tilde{f}) = f   \}
	\quad \textup{ since $\piod(M_i \alpha) = M\alpha$ for all $\alpha$} 
	\\&= 
	\max \{\tilde{\nu}_M(\tilde{f}) \, \mid \, \tilde{f} \in \C[x_1,\ldots,x_n] \textup{ and } \pi(\tilde{f}) = f \} 
	\quad \textup{ by definition of $\tilde{\nu}_M$} 
	\\ & = 
	\nu_{M} (f)
	\end{align*}
	as desired. 
\end{proof}

We can now prove Lemma~\ref{lemma: NOBY projects to NOBY}. 

\begin{proof}[Proof of Lemma~\ref{lemma: NOBY projects to NOBY}]
By Definition~\ref{definition: NOBY} we know $\Delta(A,\nu_{M_i}) = \cone(S(A,\nu_{M_i})) \cap (\{1\} \times \R^d)$ for $i=1,2$ and similarly for $\Delta(A,\nu_M)$. Since $\piod$ is a linear map, $\piod(\cone(S(A,\nu_{M_i}))) = \cone(\piod(S(A,\nu_{M_i})))$. Now by Lemma~\ref{lemma: projection of valuation} we know that $\piod(S(A,\nu_{M_i})) = S(A,\nu_M)$ for $i=1,2$. Hence, $\cone(\piod(S(A,\nu_{M_i}))) = \cone(S(A,\nu_M))$ for $i=1,2$. The projection $\piod$ preserves the first coordinate, so taking the level-$1$ slice commutes with $\piod$ and the statement follows. 
\end{proof}

We now wish to deduce a relationship between the fibers on the corresponding polytopes 
\[
\piod^{-1}(\xi) \cap \Delta(A,\nu_{M_1}) \quad \textup{ and } \quad \piod^{-1}(\xi) \cap \Delta(A,\nu_{M_2}). 
\]
for $\xi \in \Delta(A, \nu_M)$. 
An example was illustrated in Figure~\ref{fig: fibers}.
To facilitate this, we define functions $\mathcal{L}_1$ and $\mathcal{L}_2$ which record the lengths of these fibers, i.e.,  
\begin{equation}\label{eq: length of fiber function}
\mathcal{L}_i: \Delta(A,\nu_M) \to \R, \quad \xi \mapsto \len(\piod^{-1}(\xi) \cap \Delta(A,\nu_{M_i}))
\end{equation}
for $i=1,2$, where $\len$ denotes the standard Euclidean length in $\R^{d+1}$ with respect to which each standard basis vector $\varepsilon_i$, $1\leq i \leq d+1$, has length $1$. 
Since any polytope is an intersection of finitely many affine half-spaces which are defined by linear inequalities, it is clear that both $\mathcal{L}_1$ and $\mathcal{L}_2$ are piecewise-linear.\footnote{A real-valued function on a polytope $\Delta$ is \emph{piecewise linear} if $\Delta$ can be written as a finite union of polytopes, on each of which $f$ is an affine function, i.e., it is a linear function plus a global translation.} 
With this notation in place, we can state the following.

\begin{theorem} \label{theorem: fiber lengths equal}
Let $\xi \in \Delta(A,\nu_M)$. Then the Euclidean lengths of $\piod^{-1}(\xi) \cap \Delta(A,\nu_M)$ and $\piod^{-1}(\xi) \cap \Delta(A,\nu_{M_2})$ are equal, up to a global constant which is independent of $\xi$. Equivalently, there exists a global constant $\kappa> 0$ such that $\kappa \mathcal{L}_1 = \mathcal{L}_2$ as piecewise linear functions on $\Delta(A,\nu_M)$. 
\end{theorem}

To prove this, we start with some preliminary observations. First, since the $\mathcal{L}_i, i=1,2$ are piecewise linear, it is straightforward that there exists a regular subdivision of $\Delta(A,\nu_M)$ such that both $\mathcal{L}_1$ and $\mathcal{L}_2$ are affine on each cell. 
With this in mind, the following lemma shows that to prove that $\kappa \mathcal{L}_1= \mathcal{L}_2$ it suffices to check equality on a suitable subset of points in $\Delta(A,\nu_M)$. 

\begin{lemma} \label{lemma: pwl equal enough for simplices}
Let $\Delta$ be an $m$-dimensional polytope, and let $f, g: \Delta \to \R$ be piecewise-linear functions on $\Delta$. 
Suppose there exist $Q_j \subseteq \Delta$ for $1 \leq j \leq N$ for some positive integer $N$ such that $\Delta = \cup_{j=1}^N Q_j$, where each $Q_j$ is a polytope and both $f$ and $g$ are affine on $Q_j$ for each $j=1, \dots, N$. Suppose that, for each $j$, $1 \leq j \leq N$, there exist a set of $m+1$ points $\{x_{j, 1}, x_{j,2}, \cdots, x_{j, m+1}\} \subseteq Q_j$ whose convex hull is an $m$-simplex, and such that $f(x_{j, k}) = g(x_{j,k})$ for all $k, 1 \leq k \leq m+1$. Then $f=g$ on $\Delta$. In particular, to check equality of $f$ and $g$ above, it suffices to check, for each $Q_j$, the equality $f(x)=g(x)$ for $x$ in a dense subset of any open $m$-ball of positive radius contained $Q_j$. 
\end{lemma}

\begin{proof}
For the first statement, it suffices to check equality on each $Q_j$ where $\mathcal{L}_1, \mathcal{L}_2$ are affine. Choose a $j$, $1 \leq j \leq N$. Since $Q_j$ is $m$-dimensional, an affine function on $Q_j$ is determined by its values on $m+1$ affinely independent vectors in $Q_j$. Since a set of $m+1$ points whose convex hull is an $m$-simplex must be affinely independent, the result follows. For the last statement, note that any open ball contains an $m$-simplex, as long as the simplex is small enough, and it is clear that the vertices can be arranged to lie in the dense subset. 
\end{proof}

For the rest of the section we use the notation $S_i:=S(A,\nu_{M_i})$ and $S:=S(A,\nu_M)$. By assumption on the $M_i$ and $M$, the semigroups $S_i$ and $S$ are contained in $\Z^{d+1}$ and $\Z^d$ respectively. 
We will also use $\Delta(S)$ (resp. $\Delta(S_i)$)  to denote $\Delta(A,\nu_M)$ (resp. $\Delta(A,\nu_{M_i})$).
Denote by $G(S)$ (resp. $G(S_i)$) the group generated by $S$ (resp. $S_i$).
The starting point of our argument is to observe that for appropriately chosen $\xi$, the Euclidean lengths of the fibers
$\piod^{-1}(\xi) \cap \Delta(S_i)$ have a geometric interpretation; this is the content of Lemma~\ref{lemma: fibers are geometric} below. We need some preparation. Let $w_1, w_2$ be the integral vectors which were chosen before the statement of Theorem~\ref{theorem:main}.

\begin{lemma}\label{lemma: initial ideals coincide} 
 $\init_{w_i}(\init_M(I)) = \init_{M_i}(I)$ for $i=1,2$.
 \end{lemma} 
 
 \begin{proof} 
 This is immediate from \cite[Lemma 8.8]{KavehManon-published}. 
 \end{proof} 
 
  Since the cones $C_i$ are prime and maximal-dimensional by assumption, the corresponding initial ideals $\init_{M_i}(I)$ are toric ideals. Let $X_i$ for $i=1,2$ denote the corresponding Gr\"obner toric degenerations. Note that Lemma~\ref{lemma: initial ideals coincide} says that we may also realize $X_i$ as a Gr\"obner toric degeneration of $Y := Proj(\C[x_1,\ldots,x_n]/\init_M(I))$. By construction, and also by the assumptions in the special case under consideration, we know that $\init_M(I)$ is homogeneous with respect to a $\Z^d$-grading; thus, $Y$ is equipped with the action of a codimension-$1$ torus $T$, and this torus still acts on the toric degeneration $X_i$. More specifically, the full-dimensional torus acting on $X_i$ (with respect to which $X_i$ is a toric variety) contains $T$ as a subtorus. We have the following. 

\begin{lemma}\label{lemma: fibers are geometric} 
Let $\xi \in \Delta(S) \cap \Q^d$ be a rational point in the relative interior of $\Delta(S)$. Let $m \in \Z$, $m >0$ such that $m \xi \in \Z^d$. Let $i=1$ or $i=2$. Then there exists a real positive constant $\kappa_i$, independent of $\xi$, such that the length 
$\len(\piod^{-1}(\xi) \cap \Delta(S_i))$ is equal to $\kappa_i/m$ times the degree of the GIT quotient $X_i //_{m\xi} T$.  
\end{lemma}

\begin{proof} 
Let $i=1$ or $i=2$. 
We know $X_i$ is a toric variety and the moment map of the codimension-$1$ subtorus is obtained by projection of $\Delta(S_i)$ to $\Delta(S)$ via $\piod$ \cite[Section 28.3]{CdS}. 
For $m$ chosen as in the statement of the lemma, we may consider $m \xi$ as a point in $m \Delta(S)$, i.e. the $m$-scalar multiple of $\Delta(S)$. Note that since $T$ is codimension $1$, the GIT (equivalently, symplectic) quotient by $T$ will be complex $1$-dimensional and real $2$-dimensional (cf. \cite{MarWei}, \cite[Theorem 23.1]{CdS}). The degree of the GIT quotient $X_i //_{m \xi} T$ is also the symplectic volume of the symplectic quotient of $X_i$ at $m \xi$ with respect to the $m$-scalar multiple of the original $T$-moment map (\cite[Theorem 13.4.1]{CLS}, \cite[Section 30.1]{CdS}). The symplectic (GIT) quotient $X_i //_{m \xi} T$ is equipped with a residual $S^1$-action ($\C^*$-action) whose moment map image is precisely  the fiber $\piod^{-1}(\xi) \cap \Delta(S_i)$ (multiplied by $m$) \cite[Section 24.3]{CdS}. It follows that the symplectic volume of the symplectic quotient is $m$ times a normalized Euclidean length of $\piod^{-1}(\xi) \cap \Delta(S_i)$ \cite[Section 30.1]{CdS}. Here the normalization factor $\kappa_i$ depends on the index of $G(S_i) \cap \{x_1=\cdots=x_d=0\}$ in $\Z$ and is hence independent of $\xi$, as claimed. 
\end{proof}

The above lemma indicates that in order to prove Theorem~\ref{theorem: fiber lengths equal}, it suffices to show that the degrees of the two GIT quotients $X_1 //_{m \xi} T$ and $X_2 //_{m \xi} T$ are equal.  This is where we use a variation of GIT.  We have the following.

\begin{lemma}\label{lemma: degrees equal} 
$\deg(X_1 //_{m \xi} T) = \deg(X_2 //_{m \xi} T)$. 
\end{lemma} 

\begin{proof} 
We observed above that both $X_1$ and $X_2$ are Gr\"obner toric degenerations of $Y$, since $\init_{w_i}(\init_{M}(I)) = \init_{M_i}(I)$. This means that there exist flat families $\mathcal{X}_1$ and $\mathcal{X}_2$ over $\A^1$ such that the generic fibers are isomorphic to $Y$ for both $\mathcal{X}_1$ and $\mathcal{X}_2$, and the special fiber is isomorphic to $X_1$ and $X_2$ respectively. We also saw that there is an action of a codimension-$1$ torus on $Y, X_1$ and $X_2$, and it is straightforward to see that this action extends to the families $\mathcal{X}_1$ and $\mathcal{X}_2$. By \cite[Theorem 2.1.1]{Hu} we know that, for $i=1$ or $i=2$, the global GIT quotient of the entire family by $T$ at $m \xi$ is a flat family $\mathcal{X}_i //_{m \xi} T$ over $\A^1$ whose generic fiber is $Y //_{m \xi} T$ and whose special fiber is $X_i //_{m \xi} T$. 
Since the family is flat, we know $\deg(Y //_{m \xi} T) = \deg(X_i //_{m \xi} T)$. Since this equality holds for both $i=1$ and $i=2$, we conclude that $\deg(X_1 //_{m \xi} T) = \deg(X_2 //_{m \xi} T)$, as desired. 
\end{proof}

\begin{proof}[Proof of Theorem~\ref{theorem: fiber lengths equal}]
From Lemma~\ref{lemma: pwl equal enough for simplices} it suffices to check the equality of lengths at all rational points in the interior of $\Delta(S)$. Let $\xi \in \Delta(S) \cap \Q^d$ be an interior point and choose $m>0, m \in \Z$ such that $m \xi \in \Z^d$. By Lemma~\ref{lemma: fibers are geometric} we know that 
$\mathcal{L}_1(\xi) = \frac{\kappa_1}{m} \deg(X_1 //_{m \xi} T)$ and $\mathcal{L}_2(\xi) = \frac{\kappa_2}{m} \deg(X_2 //_{m \xi} T)$ where both $\kappa_1, \kappa_2$ are real and positive global constants that are independent of $\xi$. From Lemma~\ref{lemma: degrees equal} we know that the degrees of the two GIT quotients $X_1 //_{m \xi} T$ and $X_2 //_{m \xi} T$ are equal, so we conclude $\frac{1}{\kappa_1} \mathcal{L}_1(\xi) = \frac{1}{\kappa_2} \mathcal{L}_2(\xi)$. Setting $\kappa=\kappa_2/\kappa_1$ completes the proof. 
\end{proof}

\section{Wall-crossing formulas for Newton-Okounkov bodies and value semigroups}\label{sec:wall-crossing}

The main result of this section is the construction of explicit wall-crossing maps $\mathsf{S}$ (the ``shift map'') and $\mathsf{F}$ (the ``flip map'') mentioned in Theorem~\ref{theorem:main}, thus completing the proof of Theorem~\ref{theorem:main}. 
This will complete the proof of our main result, Theorem~\ref{theorem:main}.
 Since these maps are defined between the polytopes, we refer to these as the ``geometric wall-crossing'' formulas. Then, in Section~\ref{subsec: algebraic crossing}, we construct a bijective map $\Theta: S_1 \to S_2$ on the semigroups that covers the identity on $S:=S(A,\nu_{M})$ and behaves well with respect to the generators of the semigroups, in a sense to be described below (see Lemma~\ref{lemma: formula for Theta}). To distinguish the map $\Theta$ from the geometric wall-crossing maps, we refer to $\Theta$ as the ``algebraic wall-crossing map''.  It should be emphasized that the algebraic wall-crossing map $\Theta$ is not necessarily a semigroup homomorphism, and it does not necessarily arise as a restriction of a geometric wall-crossing map to the semigroup. Example~\ref{example: algebraic is not geometric} illustrates these points.

\subsection{Geometric wall-crossing for Newton-Okounkov bodies}\label{subsec: geometric wall-crossing} \label{subsec: geometric crossing}

The goal of this section is to construct the two piecewise-linear maps $\mathsf{F}$ and $\mathsf{S}$ between the Newton-Okounkov bodies $\Delta(S_1)$ and $\Delta(S_2)$ in the same setting as Section~\ref{sec:fibers}.  For the purpose of this discussion we view the polytopes $\Delta(S_1)$ and $\Delta(S_2)$ in the ``level-$1$'' affine subspace $\{1\} \times \R^d \subseteq \R^{d+1}$ as in Section~\ref{sec:fibers}.

Let $i=1$ or $2$. Since $\Delta(S_i)$ is a polytope and projects to $\Delta(S)$, there exist piecewise-linear functions $\varphi_i: \Delta(S) \to \R$ and $\psi_i: \Delta(S) \to \R$ such that 
	\begin{equation}\label{eq:NOBY-ineqs}
	\Delta(S_i) = \{ ((1,v), z) \in \{1\} \times \R^{d-1} \times \R \, \vert \, (1,v) \in \Delta(S),  \varphi_i(1,v) \leq z \leq \psi_i(1,v) \} \subseteq  \{1\} \times \R^d. 
	\end{equation}
From Theorem~\ref{theorem: fiber lengths equal} we know that for any $(1,v) \in \Delta(S)$ we have 
	\begin{equation}\label{eq:fiberlengths-equal}
	\psi_1(1,v)- \varphi_1(1,v)=
	\len \left(\mathsf{p}^{-1}(1,v) \cap \Delta(S_1)\right)=
	\frac{1}{\kappa} \, \len \left(\mathsf{p}^{-1}(1,v) \cap \Delta(S_2)\right)=
	\frac{1}{\kappa} \left( \psi_2(1,v)- \varphi_2(1,v) \right)
	\end{equation}
	where $\kappa := |\kappa_1/\kappa_2|$ is the global constant, appearing in Theorem~\ref{theorem: fiber lengths equal}, which depends on the choices of $C_i, M_i$. 
Using this, we define the \textbf{shift map} $\mathsf{S}_{12}$ by the formula
	\begin{align}\mathsf{S}_{12}:\R^{r+1}&\rightarrow\R^{r+1} \nonumber\\
	(1,v,z)&\mapsto(1,v, \kappa(z - \varphi_1(1,v)) + \varphi_2(1,v))  \label{eq: shift map}
	\end{align}
and we define the \textbf{flip map} $\mathsf{F}_{12}$ as 
	\begin{align}\mathsf{F}_{12}:\R^{r+1}&\rightarrow\R^{r+1} \nonumber \\
	(1,v,z)&\mapsto(1,v, \kappa( - z+   \varphi_1(1,v)) + \psi_2(1,v)).   \label{eq: flip map}
	\end{align}

We can now complete the proof of Theorem~\ref{theorem:main}. 

\begin{proof}[Remainder of proof of Theorem~\ref{theorem:main}] 
Since we already saw in Lemma~\ref{lemma: NOBY projects to NOBY} and Theorem~\ref{theorem: fiber lengths equal} that the first claims of Theorem~\ref{theorem:main} hold, it remains to show that the maps ${\sf S}_{12}$ and ${\sf F}_{12}$ from 
$\Delta(S_1)$ to $\Delta(S_2)$ are piecewise-linear, bijective, and that the following diagrams commute: 
\[
\xymatrix{
\Delta(S_1) \ar[rr]^{\mathsf{S}_{12}} \ar[rd]_{\piod} & & \Delta(S_2) \ar[ld]^{\piod} \\ 
 & \Delta(S) &
 }
 \quad 
 \xymatrix{ 
 \Delta(S_1) \ar[rr]^{\mathsf{F}_{12}} \ar[rd]_{\piod} & & \Delta(S_2) \ar[ld]^{\piod} \\ 
 & \Delta(S) &} 
\]

To do this, we first check that the maps are well-defined, i.e., they take values in $\Delta(S_2)$ as claimed. 
It is straightforward to check that both maps are injective. Let $(1,v,z)\in \Delta(S_1)$. 
 We have 
 \begin{equation}
 \begin{split} 
 \varphi_1(1,v) \leq z \leq \psi_1(1,v) & \iff 0 \leq z - \varphi_1(1,v) \leq \psi_1(1,v) - \varphi_1(1,v) \\
  & \iff \varphi_2(1,v) \leq \kappa(z - \varphi_1(1,v)) + \varphi_2(1,v) \leq \kappa (\psi_1(1,v) - \varphi_1(1,v)) + \varphi_2(1,v) \\
  & \iff \varphi_2(1,v) \leq \kappa(z - \varphi_1(1,v)) + \varphi_2(1,v) \leq \psi_2(1,v) - \varphi_2(1,v) + \varphi_2(1,v) \\
  & \iff \varphi_2(1,v) \leq \kappa(z-\varphi_1(1,v)) + \varphi_2(1,v) \leq \psi_2(1,v)
 \end{split} 
 \end{equation}
 where we have used the fact that $\kappa(\psi_1(1,v)-\varphi_1(1,v)) = \psi_2(1,v) - \varphi_2(1,v)$. It follows that $\mathsf{S}_{12}$ is well-defined, and the argument for $\mathsf{F}_{12}$ is similar. Since $\psi_i, \varphi_i, i=1,2$ are piecewise-linear, it follows that both $\mathsf{S}_{12}$ and $\mathsf{F}_{12}$ are piecewise linear. Similar arguments show that both are bijective, and the diagrams commute by construction. 
 This completes the proof of Theorem~\ref{theorem:main}. 
\end{proof}

We can extend the definitions of the shift and flip maps to the cones $\cone(S_1), \cone(S_2)$. This is useful when we consider the relationship between the geometric wall-crossing maps $\mathsf{S}_{12}$ and $\mathsf{F}_{12}$ with the algebraic wall-crossing map to be defined in the next section.

\begin{remark}\label{rem: maps extend to cones} 
Let $(s,v)\in \cone(S_1)$ for $s \neq 0$. By rescaling, we obtain that that $(1,\frac{1}{s}v)\in\Delta(S_1)$, since  $\cone(S_1)$ is the cone over $\Delta(S_1)$. Then ${\sf F}_{12}(1,\frac{1}{s}v)\in\Delta(S_2)$ and therefore $s\cdot{\sf F}_{12}(1,\frac{1}{s}v)\in \cone(S_2)$. A similar formula holds for $\mathsf{S}_{12}$. 
Thus we can extend the shift map \eqref{eq: shift map} and the flip map \eqref{eq: flip map} to $\cone(S_1)$ as follows: 
	\begin{align*}
	{\sf F}_{12}:	\cone(S_1)	&\to \cone(S_2)  \\
	(s,v)	&	\mapsto	s\cdot{\sf F}_{12}(1, v/s).
\end{align*}
The same holds for ${\sf S}_{12}$.
\end{remark}

\subsection{Wall-crossing for value semigroups} \label{subsec: algebraic crossing}

In the previous section, we constructed maps between the Newton-Okounkov polytopes $\Delta(S_1)$ and $\Delta(S_2)$ associated to the maximal-dimensional prime cones $C_1$ and $C_2$. In this section, 
we turn our attention to the underlying semigroups $S_1$ and $S_2$ and ask whether there exists a natural bijection $\Theta: S_1 \to S_2$ between them which would cover the identity on $S$, i.e., so that the diagram 
\begin{equation}\label{eq: diagram on semigroups}
\xymatrix{
S_1 \ar[rr]^{\Theta} \ar[rd]_{\piod} & & S_2 \ar[ld]^{\piod} \\ 
 & S &
 }
 \end{equation} 
 commutes. The answer, which is the content of this section, is that there does exist such a natural map, at least under the hypothesis that the two cones 
$C_1$ and $C_2$ are both faces of a single maximal cone $C_<$ of the Gr\"obner fan of $I$. 
Let $\mathcal{S}(<,I) \subset \C[x_1, \ldots, x_n]$ denote the set of standard monomials with respect to $I$ and the monomial order $<$ and let $b_\alpha := \pi(x^{\alpha})$ denote the projection to $A$ of $x^{\alpha} \in \mathcal{S}(<,I)$. The following is known. 

\begin{proposition}\label{prop: common adapted basis} (\cite[Proposition 3.3]{KavehManon-published}) 
Given $C$ as above, let $M$ be an $r \times n$ matrix with $j$-th row equal to $u_j$ for linearly independent vectors $\{u_1,\ldots,u_r\} \subset C$.  Then the set $\mathbb{B} := \{b_{\alpha}\}$ is an adapted basis of $A$ with respect to $\mathfrak{\nu}_M$. Moreover, we have $\init_{<}(\init_M(I)) = \init_{<}(I)$. 
\end{proposition}  

The point of the above proposition is that, if $C_1$ and $C_2$ are both faces of the same maximal cone $C_{<}$ in the Gr\"obner fan, then the same set $\mathcal{S}(<,I)$ of standard monomials with respect to $<$ projects to give an adapted basis of $A$ for both $\nu_{M_1}$ and $\nu_{M_2}$. 
This fact allows us to produce a function $S_1 \to S_2$ as follows. Applying Proposition~\ref{prop: common adapted basis} to $M_i$ for $i=1$ and $2$, we conclude that $\mathbb{B}$ is adapted to both $\nu_{M_1}$ and $\nu_{M_2}$. Since both $\nu_{M_1}$ and $\nu_{M_2}$ have one-dimensional leaves, we can conclude that the valuations $\nu_{M_i}$ for $i=1$ and $2$ induce bijections 
\[
\theta_1: \mathbb{B} \to S_1 \textup{ defined by } b_{\alpha} \mapsto \nu_{M_1}(b_\alpha)
\]
for each $b_{\alpha} \in \mathbb{B}$, and similarly 
\[
\theta_2: \mathbb{B} \to S_2 \textup{ defined by } b_{\alpha} \mapsto \nu_{M_2}(b_\alpha).
\]
Then the function on semigroups may be defined by 
\begin{equation}\label{eq: algebraic crossing} 
\Theta := \theta_2 \circ \theta_1^{-1}:   S_1 \to S_2.
\end{equation}
We refer to $\Theta$ as the \textbf{algebraic wall-crossing} map.  
Moreover, the above argument shows that this is well-defined and a bijection.

The following, which is a straightforward consequence of \cite[Lemma 2.32]{KavehManon-published}, will be computationally useful. 

\begin{lemma}\label{lemma: formula for Theta}
Let $\Theta: S_1 \to S_2$ be the map defined above and let $x^{\alpha} \in \mathcal{S}(<,I)$. 
Then $\Theta(M_1\alpha) = M_2 \alpha$.
\end{lemma}

We now show that the diagram~\eqref{eq: diagram on semigroups} commutes. Recall that the projection map $\piod: S_i \to S$ forgets the last coordinate. 

\begin{lemma}\label{lemma: algebraic covers identity} 
The map $\Theta$ covers the identity on $S$, i.e., for all $u \in S_1$, we have $\piod(u) = \piod(\Theta(u))$. 
\end{lemma}

\begin{proof} 
Since $\theta_1$ and $\theta_2$ are bijections, we know that any element in $S_1$ (respectively $S_2$) can be written as $M_1 \alpha$ (respectively $M_2\alpha$) for some $x^{\alpha} \in \mathcal{S}(<,I)$. 
Lemma~\ref{lemma: formula for Theta} implies that it suffices to show that, for all $x^{\alpha} \in \mathcal{S}(<,I)$, we have $\piod(M_1\alpha) = \piod(M_2\alpha)$. This follows immediately from the fact that $M_1$ and $M_2$ are equal except on the bottom row. 
\end{proof}

Since the map $\Theta$ defined above is a map between semigroups, it is natural to ask whether $\Theta$ is in fact a semigroup homomorphism.  Moreover, since the shift and flip maps of Section~\ref{subsec: geometric crossing} can be defined on all of $\cone(S_1)$ and the semigroup $S_1$ lies in $\cone(S_1)$, we can ask whether the restriction of either of the ``geometric'' wall-crossing maps -- i.e. the shift or the flip map -- to the subset $S_1$ is equal to $\Theta$. 
It turns out that, in general, $\Theta$ need not be a semigroup homomorphism, and 
$\Theta$ is not necessarily obtained by restriction of $\mathsf{S}_{12}$ or $\mathsf{F}_{12}$. We give an example to illustrate this.

\begin{example}\label{example: algebraic is not geometric} 
 First we illustrate that the algebraic wall-crossing map need not be the restriction of either of the geometric wall-crossing maps.

Let $f = x_2^{11} - x_1^6 x_3^4 x_4 - x_1^7 x_3 x_4^3 \in \C[x_1,x_2, x_3,x_4]$ and let $I = \langle f \rangle$ be the principal ideal generated by $f$. 
 The tropical hypersurface $\mathcal{T}(\langle f\rangle)$ is defined to be the set of $(u_1,u_2,u_3,u_4)\in\R^4$ such that $\init_u(f)$ is not monomial. It is not hard to see that 
  two of the maximal ($3$-dimensional) cones of $\mathcal{T}(\langle f\rangle)$ are given by 
	\begin{align*}C_1&=\cone\{(0,0,-1,4),\pm(1,1,1,1),\pm(0,1,2,3)\} \textup{ and }  \\
	C_2&=\cone\{(0,0,3,-1),\pm(1,1,1,1),\pm(0,1,2,3)\}.
	\end{align*}
	(There is another maximal cone $C_3$ which we do not need to consider, since it is not prime.) 
The initial terms of $f$ corresponding to the cones $C_1$ and $C_2$ above are
	\begin{align*}
	\init_{C_1}(f)&=  x_2^{11} - x_1^6 x_3^4 x_4 \textup{ and } \\
	\init_{C_2}(f)&= x_2^{11} - x_1^7 x_3 x_4^3. 
	\end{align*}
	We claim that both $\init_{C_1}(f)$ and $\init_{C_2}(f)$ are irreducible, and thus that $C_1$ and $C_2$ are maximal-dimensional prime cones in $\mathcal{T}(\langle f \rangle)$. It is clear from the above that $C_1$ and $C_2$ share a codimension-$1$ face, so this means we are in the situation being discussed in this manuscript. Since the arguments for irreducibility of $\init_{C_1}(f)$ and $\init_{C_2}(f)$ are similar, we sketch the argument only for $\init_{C_1}(f)$. Consider the matrix 
	\[
	A = \begin{pmatrix} 1 & 1 & 1 & 1 \\ 0 & 1 & 2 & 3 \\ 0 & 0 & -1 & 4 \end{pmatrix}.
	\]
	Then the kernel of $A$, considered as a linear transformation $\C^4 \to \C^4$, is spanned by the vector $(-6, 11, -4, -1)$. Define a map $\varphi: \C[x_1, x_2, x_3, x_4] \to \C[t_1^{\pm 1}, t_2^{\pm 1}, t_3^{\pm 1}]$ by $x_1 \mapsto t_1, x_2 \mapsto t_1 t_2, x_3 \mapsto t_1 t_2^2 t_3^{-1}$ and $x_4 \mapsto t_1 t_2^3 t_3^4$. Let $I_A := \mathrm{ker} \varphi$. This is a prime ideal since the image of $\varphi$ is a domain (being the subring of a domain); it is also called the toric ideal of $A$. By \cite[Exercise 3.2, Section 3.1]{HerzogHibiOhsugi}, $I_A$ is principal  since $A$ has rank $3$. It is straightforward to check that $\init_{C_1}(f)$ is contained in $I_A = \mathrm{ker} \varphi$. We now claim that $\init_{C_1}(f)$ is a minimal generator of $I_A$. 
	We need some notation. For a vector $\alpha \in \Z^4$ we define $\alpha_+$ and $\alpha_-$ by the formulas
	\[
	(\alpha_+)_i := \begin{cases}   \alpha_i \textup{ if } \alpha_i \geq 0 \\ 0 \textup{ if } \alpha_i < 0 \end{cases}
	\quad \quad 
	(\alpha_-)_i := \begin{cases} 0 \textup{ if } \alpha_i > 0 \\ -\alpha_i \textup{ if } \alpha_i \leq 0. \end{cases} 
	\]
	By \cite[Theorem 3.2]{HerzogHibiOhsugi}, there exists a binomial $f_\alpha = \underline{x}^{\alpha_+} - \underline{x}^{\alpha_{-}}$ with $\alpha \in \mathrm{ker}(A) \cap \Z^4$ such that $f_{\alpha}$ generates $I_A$ and divides $\init_{C_1}(f)$. Since the kernel of $A$ is spanned by $(-6,11,-4,-1)$, there must exist a constant $c \in \C$ such that $c(-6,11,-4,-1) = (\alpha_1, \alpha_2, \alpha_3, \alpha_4)$. 
Since $\alpha\in\Z^4$, we conclude $c$ must be an integer. If $c \neq \pm 1$, then the total degree of $f_\alpha$ is greater than $11$, so we conclude $c = \pm 1$. This then implies $f_\alpha = \pm \init_{C_1}(f)$. Thus, $\init_{C_1}(f)$ is a minimal generator of $I_A$, and since $I_A$ is prime, $\init_{C_1}(f)$ is irreducible.  A similar argument shows $\init_{C_2}(f)$ is irreducible.   We conclude that $C_1$ and $C_2$ are both maximal-dimensional prime cones.

	For the cones $C_1$ and $C_2$ we may choose the corresponding matrices $M_1$ and $M_2$ as follows 
	\[
	M_1 = \begin{pmatrix} 1 & 1 & 1 &1 \\ 0 & 1 & 2 & 3 \\ 0 & 0 & -1 & 4 \end{pmatrix} \quad \textup{ and } 
	\quad M_2 = \begin{pmatrix} 1 & 1 & 1 & 1 \\ 0 & 1 & 2 &3 \\ 0 & 0 & 3 & -1 \end{pmatrix}.
	\]
We illustrated the pair of polytopes associated to these matrices in Figure~\ref{fig: fibers}.

Notice that both $C_1$ and $C_2$ lie in the maximal cone of the Gr\"obner fan corresponding to $\init_{<}(I) = \langle x_2^{11} \rangle$. 
The standard monomials $\mathcal{S}(<,I)$ for $I$ with respect to (a choice of such) a monomial order $<$ for this Gr\"obner cone is the set of all monomials not divisible by $x_2^{11}$. 
	Since $x_1, x_2, x_3, x_4$ are all standard monomials, the
	algebraic wall-crossing map $\Theta$ sends the $j$-th column of $M_1$ to the $j$-th column of $M_2$ for $1 \leq j \leq 4$.  In particular, $\Theta(1,1,0)=(1,1,0)$, since the second column goes to the second column. The Newton-Okounkov bodies in question are the convex hulls of the columns of $M_i$, which we now view as polygons in $\R^2 \cong \{1\} \times \R^2$. Thus the point $(1,1,0)$ considered above is now identified with the point $(1,0)$ in $\R^2$, and this point is contained in the interior of both $\Delta(A,\nu_{M_1})$ and $\Delta(A,\nu_{M_2})$, see Figure~\ref{fig: alg not geom}. Moreover, we have just seen that this interior point $(1,1,0)$ in $\Delta(A,\nu_{M_1})$ must be sent by $\Theta$ to the point $(1,1,0)$ in $\Delta(A,\nu_{M_2})$. It is an easy exercise to check that neither the geometric ``flip'' map  $\mathsf{F}_{12}$ nor the geometric ``shift'' map $\mathsf{S}_{12}$ can accomplish this.   Therefore, the algebraic wall-crossing map $\Theta$ does not arise as the restriction of a geometric wall-crossing in this case. 
	
	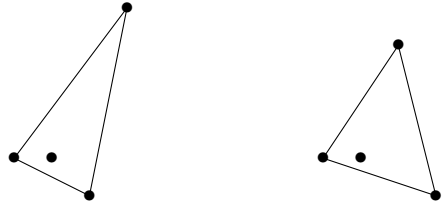
\begin{figure}[h]
	\begin{tikzpicture}\begin{scope}[scale=.5]
	\draw (0,0) node {$\bullet$}--(2,-1) node {$\bullet$}--(3,4) node {$\bullet$}--(0,0);
	\draw (1,0) node {$\bullet$};
	\end{scope}
	\end{tikzpicture}
	\qquad 
	\qquad 
	\qquad 
	\begin{tikzpicture}\begin{scope}[scale=.5]
	\draw (0,0) node {$\bullet$}--(2,3) node {$\bullet$}--(3,-1) node {$\bullet$}--(0,0);
	\draw (1,0) node {$\bullet$};
	\end{scope}
	\end{tikzpicture}
	\caption{The Newton-Okounkov bodies for the matrices $M_1$ and $M_2$ in Example~\ref{example: algebraic is not geometric}.}\label{fig: alg not geom}
	\end{figure}

Secondly, we show that for this example the algebraic wall-crossing $\Theta$ is  also not a semigroup homomorphism. We follow the notation above.
As already noted, the standard monomials of $\init_{C_1}(I)$ with respect to $<$ are all monomials not divisible by $x_2^{11}$. We have also already seen that $\Theta(1,1,0)=(1,1,0)$. 
If $\Theta$ were a semigroup map, then we must have $\Theta(11,11,0)=11 \cdot \Theta(1,1,0)$. 
However, since $x_2^{11}$ is not a standard monomial, in order to compute $\Theta$ of $(11,11,0) = 11(1,1,0) = M_1 \cdot (0,11,0,0)$ we must first find a standard monomial $x^\alpha\in\mathcal{S}(<,I)$ such that $M_1 \cdot (0,11,0,0) = M_1 \alpha$. 
Notice that $x_1^6x_3^4x_4$ accomplishes this. Therefore,
\[
\Theta(11,11,0) = \Theta(M_1 \cdot (6,0,4,1)^T) = M_2 (6,0,4,1)^T = (11,11,11) \neq 11 \cdot \Theta(1,1,0)
\]
where by slight abuse of notation we have denoted vectors occasionally as rows and at other times as columns. 
Hence we conclude that $\Theta$ is not a semigroup map.
\end{example}

\section{Example: the Grassmannian of 2-planes in $m$-space}\label{sec:Gr2m}

In this section, we illustrate the wall-crossing phenomena developed above for the tropical Grassmannian $\trop(Gr(2,m))$. 
In addition, although we saw in Example~\ref{example: algebraic is not geometric} that the algebraic wall-crossing map is not necessarily the restriction of a geometric wall-crossing, we show in Theorem~\ref{theorem: main Gr2m} that 
in the case of $\trop(Gr(2,m))$, 
the algebraic crossing $\Theta$ is the restriction of the geometric ``flip'' map.

\subsection{Background on the tropical Grassmannians}

To begin, we briefly establish some notation. Let $Gr(2,m)$ denote the Grassmannian of $2$-planes in $\C^m$ embedded in $\P(\Lambda^2(\C^m))$ via the \textbf{Pl\"ucker embedding}. 
For each subset $J \subseteq [m] := \{1,2,\ldots,m\}$ of cardinality $2$ we associate a variable $p_J$. 
It is well-known that the homogeneous coordinate ring $A$ of $Gr(2,m)$ with respect to the Pl\"ucker embedding satisfies $A \cong \C[p_J : J \subset [m], \lvert J \rvert =2]/I_{2,m}$
where $I_{2,m}$ is the \textbf{Pl\"ucker ideal} 
	\begin{equation}\label{eq: plucker ideal}
	I_{2,m}=\langle	p_{ij}p_{kl}-p_{ik}p_{jl}+p_{il}p_{jk}\mid	1\le i<j<k<l\le m	\rangle
	\end{equation}
(see e.g. \cite[Proposition 2.2.10]{MaclaganSturmfels}). 
Let us now briefly summarize some facts about the tropical Grassmannian $\mathcal{T}(I_{2,m}) = \mathrm{trop}(\tilde{Gr}^0(2,m))$; see \cite{SpeyerSturmfels,MaclaganSturmfels}. We need some terminology. 
A \textbf{phylogenetic tree} on $[m]$ is a tree with $m$ labelled leaves and no vertices of degree $2$. The $m$ edges which are adjacent to the leaves of the tree are called \textbf{pendant edges} and the others are called \textbf{interior edges}. 
Given a phylogenetic tree $\tree$ on $[m]$, a \textbf{tree distance} is a vector $\underline{d} = (d_{ij}) \in \R^{\binom{m}{2}}$ constructed as follows. Assign a length $\ell_\varepsilon \in \R$ to each edge $\varepsilon$ in $\tree$ (note we do not assume the lengths are positive). Since $\tree$ is a tree, there is a unique path connecting any two leaves $i$ and $j$; let $d_{ij}$ be the sum of the lengths $\ell_\varepsilon$ of all the edges in this path.
The set of all tree distances in $\R^{\binom{m}{2}}$ is called the \textbf{space of phylogenetic trees}.  

\begin{theorem}\label{thm:phyloTrees} (\cite[Theorem 4.3.5]{MaclaganSturmfels}) 
The negative $- \mathcal{T}(I_{2,m}) \subseteq \R^{\binom{m}{2}}$ of the tropical Grassmannian is equal to the space of phylogenetic trees with $m$ labelled leaves.
\end{theorem}

We now briefly describe the fan structure of $\mathcal{T}(I_{2,m}) \subseteq \R^{\binom{m}{2}}$. For details and proofs see \cite{MaclaganSturmfels}. 
The maximal cones of $\mathcal{T}(I_{2,m})$ are in bijective correspondence with the set of trivalent trees on $[m]$, where a tree is \textbf{trivalent} if all the interior vertices are incident to exactly three edges.
We label the coordinates in $\R^{\binom{m}{2}}$ by the subsets $J$ of $[m]$ of cardinality $2$, corresponding naturally to the Pl\"ucker coordinates $p_J$. For such a subset $J$, let $e_J$ denote the standard basis (``indicator'') vector with a $1$ in the coordinate labelled by $J$ and $0$'s elsewhere.
The lineality space \footnote{The lineality space of $\mathcal{T}(I)$ for an ideal $I$ is the subspace of $w \in \R^n$ such that $\init_w(I)=I$.}
 $L$ is given by 
 \[
 L = \mathrm{span}\left( \sum_{J :i \in J} e_J \, \mid \, 1 \leq i \leq m \right)
 \]
 and this $m$-dimensional subspace is contained in all cones of $\mathcal{T}(I_{2,m})$. 
 The ideal $I_{2,m}$ is a homogeneous ideal with respect to the usual $\Z$-grading where $\deg(p_J)=1$ for each Pl\"ucker coordinate $p_J$, so we also note that $L$ contains the vector $\mathbb{1} := (1,1,1,\ldots,1) \in \R^{\binom{m}{2}}$.
Next, let $\tree$ be a trivalent tree and $\varepsilon$ be an edge of $\tree$. 
The choice of $\varepsilon$ naturally yields a partition of the $m$ leaves into two subsets $J_\varepsilon$ and $J_\varepsilon^c$, given by the decomposition of the vertices obtained by removing $\varepsilon$. We can define a corresponding tree distance
	\begin{equation}\label{eq: interior edge tree distance}
	\underline{d}_\varepsilon:= \sum_{i \in J_\varepsilon, j \in J_\varepsilon^c} e_{ij},
	\end{equation} 
obtained by assigning length $1$ to the edge $\varepsilon$.
By Theorem~\ref{thm:phyloTrees}, $-\underline{d}_\varepsilon\in\mathcal{T}(I_{2,m})$.
However, we would like to keep the entries positive.
In the case in which $\epsilon$ is a pendant edge, since $-\underline{d}_\varepsilon\in L$ we will use $\underline{d}_\varepsilon$ instead.
In the case in which $\epsilon$ is an interior edge, since $\mathbb{1}\in L$ we will use the vector $\mathbb{1}-\underline{d}_\varepsilon\in\mathcal{T}(I_{2,m})$ instead.
The maximal cone $C_\tree$ corresponding to such a tree $\tree$ is isomorphic to $\R_{\geq 0}^{m-3} \times \R^m$ and can be described explicitly as 
\[
C_\tree = \cone\left\{ \mathbb{1}-\sum_{i \in J_\varepsilon, j \in J_\varepsilon^c} e_{ij} \, \bigg\vert \, \textup{ $\varepsilon$ an interior edge } \right\} \times \mathrm{span}\left\{ \sum_{J : i \in J} e_J \, \bigg\vert \, 1 \leq i \leq m \right\} \cong \R_{\geq 0}^{m-3} \times \R^m
\]
where $\cone$ denotes the non-negative span of the given set of vectors, and $\mathrm{span}$ denotes the usual $\R$-span \cite[Proposition 4.3.10]{MaclaganSturmfels}.

\subsection{Newton-Okounkov bodies of adjacent maximal-dimensional prime cones in $\mathcal{T}(I_{2,m})$}

We now describe the Newton-Okounkov bodies and value semigroups associated to adjacent maximal-dimensional prime cones in $\mathcal{T}(I_{2,m})$. 
To begin, we need to know the set of maximal-dimensional prime cones in $\mathcal{T}(I_{2,m})$. The following is known. 

\begin{lemma} 	(\cite[Remark 4.3.11]{MaclaganSturmfels})\label{lem:allprime}
 Let $\tree$ be a trivalent tree on $[m]$ and let $C_\tree$ be the associated cone in $\mathcal{T}(I_{2,m})$. Then the initial ideal $\init_{C_\tree}(I_{2,m})$ corresponding to $C_\tree$ is a prime ideal. Equivalently, all the maximal cones of $\mathcal{T}(I_{2,m})$ are prime in the sense of Definition~\ref{def: prime cone}. 
 \end{lemma}

 From Lemma~\ref{lem:allprime} it follows that we can apply Theorem~\ref{theorem:kavehmanon theorem 4} to any maximal cone $C_\tree$ in $\mathcal{T}(I_{2,m})$. Doing so involves an explicit choice of linearly independent vectors in the relevant cones. We wish to describe the Newton-Okounkov bodies concretely and also to compare the Newton-Okounkov bodies of adjacent maximal-dimensional prime cones, so to facilitate our computations, we will make a systematic choice of these vectors.

We begin by characterizing adjacency of the maximal-dimensional prime cones. 
It is known that two maximal-dimensional prime cones $C_{\tree_1}$ and $C_{\tree_2}$ are adjacent exactly if there exists an interior edge in $\tree_1$ and an interior edge in $\tree_2$, such that we obtain the same tree after contracting these edges in their corresponding tree. Figure~\ref{fig: tree flip} shows what this looks like locally.

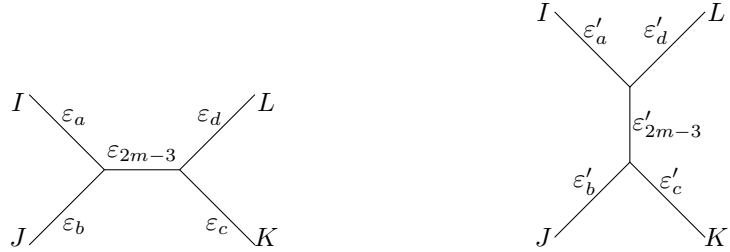
\begin{figure}[h] 
	\begin{tikzpicture}
	\draw (0,1)--(1,0)--(0,-1);
	\draw (1,0)--(2,0);
	\draw (3,1)--(2,0)--(3,-1);
	\node at (1.5,.2) {$\varepsilon_{2m-3}$};
	\node at (.6,.7) {$\varepsilon_a$};
	\node at (.6,-.75) {$\varepsilon_b$};
	\node at (2.4,.7) {$\varepsilon_d$};
	\node at (2.5,-.75) {$\varepsilon_c$};
	\node at (-.15,.9) {$I$};
	\node at (-.15,-.9) {$J$};
	\node at (3.15,-.9) {$K$};
	\node at (3.15,.9) {$L$};
	\end{tikzpicture}
	\hspace{3cm}
	\begin{tikzpicture}
	\draw (0,1)--(1,0)--(2,1);
	\draw (1,0)--(1,-1);
	\draw (0,-2)--(1,-1)--(2,-2);
	\node at (1.5,-.5) {$\varepsilon'_{2m-3}$};
	\node at (.55,.75) {$\varepsilon'_a$};
	\node at (.4,-1.25) {$\varepsilon'_b$};
	\node at (1.35,.75) {$\varepsilon'_d$};
	\node at (1.55,-1.25) {$\varepsilon'_c$};
	\node at (-.15,1) {$I$};
	\node at (-.15,-2) {$J$};
	\node at (2.15,-2) {$K$};
	\node at (2.15,1) {$L$};
	\end{tikzpicture}
	\caption{The figure on the left schematically represents $\tree_1$ and the right figure represents $\tree_2$. If $\tree_1$ and $\tree_2$ are adjacent, then they are identical except on one interior edge; in the figure these are labelled $\varepsilon_{2m-3}$ and $\varepsilon'_{2m-3}$. There also exists a decomposition $I \sqcup J \sqcup K \sqcup L = [m]$ of the leaves such that the trees schematically look as above, where the edge $\varepsilon_i$ leading to $I$ indicates that the vertices to which $\varepsilon_i$ leads lie precisely in $I \subseteq [m]$, and similarly for the others. It is understood that $\tree_1$ and $\tree_2$ are identical except near the edges $\varepsilon_{2m-3}$ and $\varepsilon'_{2m-3}$.} 
	
	\label{fig: tree flip} 
	\end{figure}

Now suppose $C_{\tree_1}$ and $C_{\tree_2}$ are adjacent maximal-dimensional prime cones. Fix $\tree \in \{\tree_1, \tree_2\}$. We choose linearly independent vectors $u_1,u_2,\ldots,u_{2m-3} \in C_{\tree}$ as follows. 
 For the purposes of this discussion, we assume that the edges of $\tree$ are labelled $\{\varepsilon_1, \varepsilon_2,\ldots,\varepsilon_{2m-3}\}$ where the first $m$ edges $\varepsilon_1,\ldots,\varepsilon_m$ are the pendant edges incident to the leaves labelled $1,2,\ldots, m$ respectively, the last $m-3$ edges $\varepsilon_{m+1},\ldots,\varepsilon_{2m-3}$ are the interior edges, and moreover, the very last interior edge $\varepsilon_{2m-3}$ (for both $C_{\tree_1}$ and for $C_{\tree_2}$) corresponds to ``the'' edge by which the two trees differ, as in Figure~\ref{fig: tree flip} above. 
 As discussed in Section~\ref{sec:fibers}  we always choose $u_1 = \mathbb{1} := (1,1,1,\ldots,1)$ so that the corresponding weight valuation is homogeneous with respect to the (usual) degree. 
 Next, for $2 \leq i \leq m$ we choose $u_i$ to be the tree distance $\underline{d}_{\varepsilon_i}=\sum_{i \in J} e_J$.
 Finally, for $m+1 \leq a \leq 2m-3$, we let $u_a =\mathbb{1}-\underline{d}_{\varepsilon_a}$. We then obtain a $(2m-3) \times \binom{m}{2}$ matrix $M_{\tree}$ whose $i$-th row is the vector $u_i$. 
By construction, $M_{\tree_1}$ and $M_{\tree_2}$ are identical except on the last (bottom) row.

\begin{example}\label{example: Gr24}
Let $m=4$. In this case the Pl\"ucker coordinates for $Gr(2,4)$ are given by the $6 = \binom{4}{2}$ coordinates $p_{12}$, $p_{13}$, $p_{14}$, $p_{23}$, $p_{24}$, $p_{34}$; throughout this discussion we assume that these $6$ Pl\"ucker coordinates are ordered as in the list just given. Let $\tree_1$ (resp. $\tree_2$) be the trivalent tree on the LHS (resp. RHS) in Figure~\ref{fig: trivalent for Gr24}. 

\begin{figure}[h] 
\begin{tikzpicture}
\begin{scope}[scale=0.5]
	\draw (0,1)--(1,0)--(0,-1);
	\draw (1,0)--(2,0);
	\draw (3,1)--(2,0)--(3,-1);
	\node at (-.2,1) {$1$};
	\node at (-.2,-1) {$2$};
	\node at (3.2,1) {$4$};
	\node at (3.2,-1) {$3$};
	\node at (1.5,.3) {$\varepsilon_5$};
\end{scope}
\end{tikzpicture}
\hspace{2cm} 
\begin{tikzpicture}
\begin{scope}[scale=0.5]
	\draw (0,1)--(1,0)--(2,1);
	\draw (1,0)--(1,-1);
	\draw (0,-2)--(1,-1)--(2,-2);
	\node at (-.2,1) {$1$};
	\node at (-.2,-2) {$2$};
	\node at (2.2,-2) {$3$};
	\node at (2.2,1) {$4$};
	\node at (1.4,-.5) {$\varepsilon'_5$};
\end{scope}
	\end{tikzpicture}
\caption{Two trivalent trees for $\mathcal{T}(I_{2,4})$.}
\label{fig: trivalent for Gr24} 
\end{figure}
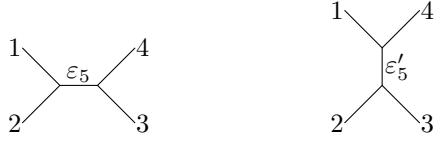 
The interior edge of $\tree_1$ partitions the set $[4]$ into the subsets $J=\{1,2\}$ and $J^c=\{3,4\}$, so $u_5 = \mathbb{1}- \sum_{i \in J, j \in J^c} e_{ij} = \mathbb{1}- (e_{13}+e_{14}+e_{23}+e_{24}) = (1,0,0,0,0,1)$. 
For $\tree_2$ the partition is $J=\{1,4\}$ and $J^c = \{2,3\}$. The matrices can be computed to be

	\[
M_{\tree_1} = \begin{pmatrix} 
1 & 1 & 1 & 1 & 1 & 1 \\
1 & 0 & 0 & 1 & 1 & 0 \\
0 & 1 & 0 & 1 & 0 & 1 \\
0 & 0 & 1 & 0 & 1 & 1 \\
1 & 0 & 0 & 0 & 0 & 1\\ 
\end{pmatrix} 
\quad \quad \textup{ and } \quad \quad 
M_{\tree_2}=\begin{pmatrix} 1 & 1 & 1 & 1 & 1 & 1 \\
1 & 0 & 0 & 1 & 1 & 0 \\
0 & 1 & 0 & 1 & 0 & 1 \\
0 & 0 & 1 & 0 & 1 & 1 \\
	0 & 0 & 1 & 1 & 0 & 0 \\ 
	\end{pmatrix}
	\]
so they are identical except on the last row. 	
\end{example}

By Theorem~\ref{theorem:kavehmanon theorem 4} and~\eqref{eq: NOBY as convex hull} we know that for any $\tree$ we have 
	\begin{align*}
	P_\tree&:=P(A, \nu_{M_\tree})=\R_{\geq 0}\textup{-span of the columns of } M_\tree, \\
	\Delta_\tree&:=\Delta(A, \nu_{M_\tree}) = \textup{ convex hull of the columns of } M_\tree.
	\end{align*}
Note that $\Delta_\tree=P_\tree\cap\{x_1=1\}$.
Furthermore, denoting by $M_{\tree_1\tree_2}$ the matrix resulting from deleting the bottom row of $M_{\tree_1}$ we also have 
	\begin{align*}
	P_{\tree_1\tree_2}&:=P(A, \nu_{M_{\tree_1 \tree_2}}) = \R_{\geq 0}\textup{-span of the columns of } M_{\tree_1\tree_2}, \\
	\Delta_{\tree_1\tree_2}&:= \Delta(A, \nu_{M_{\tree_1 \tree_2}}) = \textup{ convex hull of the columns of } M_{\tree_1\tree_2}.
	\end{align*}

\subsection{The geometric wall-crossing maps for $Gr(2,m)$}\label{subsec: geometric Gr2m}\label{subsec: geometric crossing Gr2m}

In this section, we describe the geometric wall-crossing maps for $Gr(2,m)$ for two adjacent maximal-dimensional prime cones $C_1$ and $C_2$ corresponding to trivalent trees $\tau_1$ and $\tau_2$. Let $M_{\tau_1}$ and $M_{\tau_2}$ denote the corresponding choices of matrices described in the previous section.

To proceed, it will be convenient to first give the inequalities which cut out the cone $P_\tau$ for a given trivalent tree $\tau$. 
In order to do so, we make a change of coordinates $\gamma: \R^{2m-3} \to \R^{2m-3}$ which transforms $P_\tree$ to a cone $\widetilde{P}_{\tree}$. It will turn out that $\widetilde{P}_{\tree}$ is more compatible with the combinatorics of phylogenetic trees, and moreover, the inequalities defining $\widetilde{P}_\tree$ are known from the work of Nohara and Ueda \cite{NoharaUeda}. 

We begin by explicitly defining the polytope $\widetilde{P}_\tree$; from this we can deduce the transformation $\gamma$. For a trivalent tree $\tree$ we define a $(2m-3) \times \binom{m}{2}$ matrix $\widetilde{M}_\tree$ 
by taking its $a$-th row to be the tree distance $\underline{d}_{\varepsilon_a}$ obtained by assigning $1$ to edge $\varepsilon_a$ and $0$ elsewhere. Labelling columns of $\widetilde{M}_\tree$ by pairs of leaves $\{i,j\}$ and rows by $a$, the matrix entries $c^{ij}_a$ of $\widetilde{M}_\tree$ can then be seen to satisfy
\begin{equation}\label{eq: cijk def} 
	c^{ij}_a=\begin{cases} 
	1,\quad \text{ if the (unique) path from $i$ to $j$ contains edge } \varepsilon_a\\
	0,\quad \text{ otherwise. }
	\end{cases}
	\end{equation}

We now define
\begin{equation}\label{eq: def tildePtree}
\widetilde{P}_\tree := \cone \{c^{ij} \mid 1 \leq i < j \leq m\} \subseteq \R^{2m-3},
\end{equation} 
i.e. $\widetilde{P}_\tree$ is the cone spanned in $\R^{2m-3}$ by the columns of $\widetilde{M}_\tree$. Similarly we define 
\begin{equation}\label{eq: def tildeDelta} 
\widetilde{\Delta}_\tree :=   \textup{ convex hull of the columns of } \widetilde{M}_\tree
\end{equation}
and
\begin{equation}\label{eq: def tildeS} 
\widetilde{S}_\tree :=   \textup{semigroup generated by the columns of } \widetilde{M}_\tree.
\end{equation}

We have the following.

\begin{lemma}\label{lemma: gamma} 
The linear map $\gamma: \R^{2m-3} \to \R^{2m-3}$ defined by 
\begin{align}
	\gamma:\R^{2m-3}&\to\R^{2m-3} \label{eq: gamma}\\
	(z_1,\ldots,z_{2m-3})&\mapsto\left(\frac{1}{2}(z_1+\cdots+z_m),z_2,\ldots,z_m,\frac{1}{2}(z_1+\cdots+z_m)-z_{m+1},\ldots,\frac{1}{2}(z_1+\cdots+z_m)-z_{2m-3}\right) \nonumber
	\end{align}
	is a linear isomorphism and 
maps the $ij$-th column of $\widetilde{M}_{\tree}$ to the $ij$-th column of $M_{\tree}$. In particular, $\gamma$
restricts to bijections $\widetilde{S}_\tree \to S_\tree$ and $\widetilde{P}_\tree\rightarrow P_\tree$.
\end{lemma}

\begin{proof} 
Recall that our convention is to order the edges so that $\varepsilon_1, \varepsilon_2, \cdots, \varepsilon_m$ are the pendant edges, with edge $\varepsilon_i$ adjacent to leaf $i$, and the edges $\varepsilon_{m+1},\cdots, \varepsilon_{2m-3}$ are the interior edges. In order to show that $\gamma$ takes $c^{ij}$ to the corresponding $ij$-th column of $M_\tree$, we check each coordinate of $\gamma(c^{ij})$. 

Fix a column $c^{ij}$. First, we consider the first coordinate. By definition, each column in $M_\tree$ has first entry equal to $1$. Therefore, to show that $\gamma(c^{ij})$ agrees with the corresponding column in $M_\tree$, we must show that the function $\frac{1}{2}(z_1+\cdots+z_m)$ (here the $z_k$ denote the standard coordinate functions in $\R^{2m-3}$) evaluates to $1$ on $c^{ij}$. By~\eqref{eq: cijk def} we see that the $i$-th and $j$-th coordinates of $c^{ij}$, corresponding to the pendant edges $\varepsilon_i$ and $\varepsilon_j$ respectively, are equal to $1$, since these edges are contained in the path connecting $i$ and $j$. Moreover, no other pendant edge is contained in this path, so all the other coordinates corresponding to pendant edges are equal to $0$. Therefore, $z_1+z_2+\cdots + z_m=2$ on $c^{ij}$ and hence $\frac{1}{2}(z_1+\cdots+z_m) = 1$, as desired. 

Second, we consider the coordinates corresponding to the pendant edges $\varepsilon_a$ for $2 \leq a \leq m$. By definition, the $a$-th row of $M_\tree$ is the tree distance $\underline{d}_{\varepsilon_a}$, which is equal to the $a$-th row of $\widetilde{M}_\tree$. Hence the entries are in fact equal, so the identity map on those coordinates, namely $z_2, \ldots, z_m$, takes the corresponding entries of $c^{ij}$ to those of the columns of $M_\tree$ as desired. 

Finally, consider the coordinates corresponding to interior edges, i.e. the $a$-th coordinates for $m+1 \leq a \leq 2m-3$. From the construction of $M_\tree$ we know that the $a$-th entry of the $ij$-th column of $M_\tree$ is $1-c^{ij}_a$. Therefore we need to show that $\frac{1}{2}(z_1+z_2+\cdots+z_m) - z_a$ evaluates on $c^{ij}$ to $1-c^{ij}_a$. But we already saw above that $\frac{1}{2}(z_1+z_2+\cdots+z_m)=1$ on $c^{ij}$, so the claim follows. 

This shows that $\gamma$ takes the columns of $\widetilde{M}_\tree$ to the corresponding columns of $M_\tree$, as desired. The second claim of the lemma follows immediately from the definitions of $\widetilde{P}_\tree$ and $P_\tree$. 
\end{proof}

\begin{remark} The proof of the lemma above shows also that $\widetilde{\Delta}_\tree$ is the intersection of $\widetilde{P}_\tree$ with the hyperplane 
\[
\left\{ \frac{1}{2} (z_1 + z_2 + \cdots + z_m) = 1\right\} = \gamma^{-1}(\{z_1=1\}).
\]
It follows that $\gamma$ also restricts to a bijection $\gamma: \widetilde{\Delta}_\tree \to \Delta_\tree$ which can be written explicitly as $(z_1,\ldots, z_{2m-3}) \mapsto (1, z_2, \ldots, z_m, 1-z_{m+1}, \ldots, 1-z_{2m-3})$. 
\end{remark}

In order to give the inequality description of $P_\tree$ it now suffices to give an inequality description of $\widetilde{P}_\tree$ and then to translate this back to $P_\tree$ using the coordinate change $\gamma$. In fact, the half-spaces defining $\widetilde{P}_\tree$ were given by Nohara and Ueda. We have the following, which follows from \cite[Theorem 4.9]{NoharaUeda}.

\begin{theorem}\label{thm:inequalities} 
The polytope $\widetilde{\Delta}_\tree$ is the intersection of the half-spaces defined by the inequalities 
	\begin{align}\label{eq:eqhyperplane}
	&z_1+\cdots+z_m =2\qquad \text{ and } \\\label{eq:ineqs}
	&\lvert z_b- z_c \rvert \le z_a\le z_b+z_c,
	\end{align}
where ${\varepsilon_a}, {\varepsilon_b}, {\varepsilon_c}$ are incident to a single interior vertex of $\tree$, and these inequalities run over all interior vertices of $\tree$. 
The cone $\widetilde{P}_\tree$ is defined as the intersection of the inequalities of \eqref{eq:ineqs}.
\end{theorem}

\begin{proof} 
As mentioned above, the statement of the theorem is essentially that of \cite[Theorem 4.9]{NoharaUeda}. However, a change of coordinates is required to deduce the above statement from \cite{NoharaUeda} so we explain this briefly here. For details we refer the reader to \cite{NoharaUeda}. In \cite[Section 6]{NoharaUeda}, the authors give a set of lattice points in $\R^{2m-3}$ whose convex hull is a polytope which they denote as $\Delta_\Gamma$. In \cite[Section 4]{NoharaUeda}, the authors perform a change of coordinates \cite[Equation (4.2)]{NoharaUeda}, and it is not hard to see that, under this change of coordinates, the lattice points whose convex hull is $\Delta_\Gamma$ get mapped to $m/2$ times the columns of our matrix $\widetilde{M}_\tree$. Therefore, under the change of coordinates \cite[Equation (4.2)]{NoharaUeda}, the polytope $\Delta_\Gamma$ of Nohara and Ueda is mapped to $\frac{m}{2}\widetilde{\Delta}_\tree$.
The equations of \cite[Theorem 4.9]{NoharaUeda} describe the inequalities of $\frac{m}{2}\widetilde{\Delta}_\tree$ as a subset of the hyperplane $z_1+\cdots+z_{2m-3}=m/2$ and therefore the cone $\widetilde{P}_\tree$ is the intersection of the inequalities of \eqref{eq:ineqs}.
The claim about $\widetilde{\Delta}_\tree$ now follows straightforwardly.
\end{proof}

We can now give explicit formulas for the geometric wall-crossing maps as in Section~\ref{sec:wall-crossing} for $Gr(2,m)$. Let $\tree_1$ and $\tree_2$ be two trivalent trees corresponding to two adjacent maximal-dimensional prime cones. The trees $\tree_1$ and $\tree_2$ agree everywhere except near one edge which we may take to be labelled as $\varepsilon_{2m-3}$, and that locally near $\varepsilon_{2m-3}$ the trees $\tree_1$ and $\tree_2$ are of the form given in Figure~\ref{fig: tree flip}. More specifically, we assume that $\tree_1$ looks locally near $\varepsilon_{2m-3}$ like the figure on the left in Figure~\ref{fig: tree flip} and $\tree_2$ is the figure on the right. Let $\widetilde{P}_{\tree_1 \tree_2}$ denote the projection of $\widetilde{P}_{\tree_1}$ (equivalently $\widetilde{P}_{\tree_2}$) to $\R^{2m-4}$, obtained by forgetting the last coordinate. Then, as in~\eqref{eq:NOBY-ineqs}, we may express $\widetilde{P}_{\tree_i}$ for $i=1,2$ as follows: 
	\begin{align*}
	&\widetilde{P}_{\tree_i} = \{ (v,z_{2m-3}) \, \vert \, v \in \widetilde{P}_{\tree_1\tree_2}, \,  \widetilde{\varphi}_i(v) \leq z_{2m-3} \leq \widetilde{\psi}_i(v) \} \\
	\end{align*} 
	for certain affine functions $\widetilde{\varphi}_i$ and $\widetilde{\psi}_i$. We have the following. 
	
\begin{lemma} 
In the setting above, we have 
	\begin{align*}
	 &\widetilde{\varphi}_1(v)=\max\{|z_a-z_b|,|z_c-z_d|\}, \qquad \widetilde{\psi}_1(v)=\min\{z_a+z_b,z_c+z_d\},\\
	 &\widetilde{\varphi}_2(v)=\max\{|z_a-z_d|,|z_b-z_c|\},\qquad \widetilde{\psi}_2(v)=\min\{z_a+z_d,z_b+z_c\}.
	\end{align*}
\end{lemma} 

\begin{proof} 
We prove the formulas for $\widetilde{\varphi}_1$ and $\widetilde{\psi}_1$. The proof for $i=2$ is similar. 
It may be helpful to refer to Figure~\ref{fig: tree flip}. 
From Theorem~\ref{thm:inequalities} we know 
$\lvert z_a - z_b \rvert \leq z_{2m+3}$, and similarly $\lvert z_d - z_c \rvert \leq z_{2m+3}$. 
We conclude $z_{2m+3} \geq \mathrm{max} \{\lvert z_a - z_b \rvert, \lvert z_d - z_c \rvert \}$.  
The other inequality in Theorem~\ref{thm:inequalities} 
 immediately imply $z_{2m+3} \leq \mathrm{min} \{ z_a+z_b, z_d+z_c \}$. This yields the desired formulas. 
\end{proof}

We can now deduce that the lengths of the fibers are equal.

\begin{lemma} 
For all $v \in \widetilde{\Delta}_{\tree_1\tree_2}$
	$$
	\widetilde{\psi}_1(v)-\widetilde{\varphi}_1(v)=\widetilde{\psi}_2(v)-\widetilde{\varphi}_2(v)
	$$
and therefore
	$$
	\text{length of } \left(\mathsf{p}^{-1}(v) \cap \widetilde{\Delta}_{\tree_1}\right)=
	\text{length of } \left(\mathsf{p}^{-1}(v) \cap \widetilde{\Delta}_{\tree_2}\right).
	$$
	\end{lemma}

	\begin{proof} 
	A computation verifies that for $\alpha,\beta,\gamma,\delta$ real numbers, we have 
	\begin{align*}
	\mathrm{min}(\alpha+\beta, \gamma+\delta) - \mathrm{max}&( \lvert  \alpha-\beta \rvert, \lvert \gamma-\delta \rvert) \\
	&= 
	\mathrm{min}(2\alpha, 2\beta, 2\gamma, 2\delta, \alpha+\beta+\gamma-\delta, \alpha+\beta-\gamma+\delta, \alpha-\beta+\gamma+\delta, -\alpha+\beta+\gamma+\delta).
	\end{align*}
	Applying the above formula to both $\mathrm{min}(z_a+z_b, z_c +z_d) - \mathrm{max}(\lvert z_a-z_b \rvert, \lvert z_c - z_d \rvert)$ and to $\mathrm{min}(z_a+z_d,  z_b +z_c) - \mathrm{max}(\lvert z_a-z_d \rvert, \lvert z_b- z_c \rvert)$ yields the result. 
	\end{proof}

	In particular, the above shows that, in this case of $Gr(2,m)$, the constant $\kappa$ appearing in Theorem~\ref{theorem: fiber lengths equal} is equal to $1$. 
Following \eqref{eq: shift map}, we can now compute that the shift map in this case is 
	\begin{align*}\widetilde{\mathsf{S}}_{12}:\R^{2m-3}&\rightarrow\R^{2m-3}\\
	(z_1,\ldots,z_{2m-3})&\mapsto(z_1,\ldots,z_{2m-2},z_{2m-3}+\max(|z_a-z_d|,|z_b-z_c|)-\max(|z_a-z_b|,|z_c-z_d|)),
	\end{align*}
and by \eqref{eq: flip map} the flip map is
	\begin{align}\widetilde{\mathsf{F}}_{12}:\R^{2m-3}&\rightarrow\R^{2m-3} \label{eq: tilde flip}\\
	(z_1,\ldots,z_{2m-3})&\mapsto(z_1,\ldots,z_{2m-2},-z_{2m-3}+\min(z_a+z_d,z_b+z_c)+\max(|z_a-z_b|,|z_c-z_d|).\nonumber
	\end{align}
In fact, it is not hard to see that the same formulas extend to give maps on the cones $\widetilde{P}_{\tree_1} \to \widetilde{P}_{\tree_2}$.

We can now describe the flip and shift maps on the original polytopes (respectively cones) $\Delta_{\tree_1}, \Delta_{\tree_2}$ (respectively $P_{\tree_1}, P_{\tree_2}$) by translating via the change of coordinates $\gamma$. Specifically, the flip map $\mathsf{F}_{12}: P_{\tree_1} \to P_{\tree_2}$ and the shift map $\mathsf{S}_{12}: P_{\tree_1} \to P_{\tree_2}$ are given by the formulas
\begin{equation}\label{eq: flip to flip}
\mathsf{F}_{12} := \gamma \circ \widetilde{\mathsf{F}}_{12} \circ \gamma^{-1} \qquad 
\textup{and} \qquad \mathsf{S}_{12} := \gamma \circ \widetilde{\mathsf{S}}_{12} \circ \gamma^{-1}
\end{equation} 
such that the following diagram commutes: 
\begin{center}
\begin{tikzpicture}
\node (p1) at (0,0) {$P_{\tree_1}$};
\node (p2) at (2,0) {$P_{\tree_1}$};
\node (cal1) at (0,-1.5) {$\widetilde{P}_{\tree_1}$}; 
\node (cal2) at (2,-1.5) {$\widetilde{P}_{\tree_2}$}; 
  \draw [dashed,->] (p1) -- (p2);
    \draw[->] (p1)--(cal1);
        \draw[<-] (p2)--(cal2);
        \draw[->] (cal1)--(cal2);
\node at (-.28,-.7) {$\gamma^{-1}$};
\node at (2.15,-.75) {$\gamma$};
\node at (1,.25) {${\sf F}_{12}$};
\node at (1,-1.25) {$\widetilde{{\sf F}}_{12}$};
\end{tikzpicture}
\end{center}
and an analogous diagram commutes for $\mathsf{S}_{12}$.

\begin{remark}
 In \cite[Proposition~3.5]{NoharaUeda} the authors describe a wall-crossing formula for $ \widetilde{\Delta}_{\tree}$.
This map agrees with the shift map $\widetilde{\mathsf{S}}_{12}$.
\end{remark}

\begin{example}\label{ex: ineqs and gemo maps}  Let $m=4$ and $\tree_1$ be the tree of Figure~\ref{fig: trivalent for Gr24}. The corresponding matrix is
	$$
	\widetilde{M}_{\tree_1}
	=
	\begin{pmatrix} 
	1 & 1 & 1 & 0 & 0 & 0 \\
	1 & 0 & 0 & 1 & 1 & 0 \\
	0 & 1 & 0 & 1 & 0 & 1 \\
	0 & 0 & 1 & 0 & 1 & 1 \\
	0 & 1 & 1 & 1 & 1 & 0
	\end{pmatrix}. 
	$$
The inequality description of $\widetilde{\Delta}_{\tree_1}$ is 
	$$
	2=z_1+z_2+z_3+z_4,
	$$
	$$
	z_1\le z_2+z_5,\quad
	z_2\le z_1+z_5,\quad
	z_5\le z_1+z_2,
	$$
	$$
	z_3\le z_4+z_5,\quad
	z_4\le z_3+z_5,\quad
	z_5\le z_3+z_4.
	$$
Using the map $\gamma$ we can now obtain the inequality description of $\Delta_{\tree_1}$. Specifically, 
since $2=z_1+z_2+z_3+z_4$ we obtain $y_1 = \frac{1}{2}(z_1+z_2+z_3+z_4)=1$; we also have $y_i=z_i$ for $i=2,3,4$ and $y_5 = y_1 - z_5 = 1-z_5$. Making appropriate substitutions in the above inequalities we obtain the inequalities for $\Delta_{\tree_1}$: 
	$$
	1=y_1,
	$$
	$$
	y_5+1\le 2y_2+y_3+y_4,\quad 2y_2+y_3+y_4+y_5\le 3,\quad y_3+y_4\le 1+y_5
	$$
	$$
	y_3+y_5\le1+y_4,\quad y_4+y_5\le 1+y_3,\quad 1\le y_3+y_4+y_5.
	$$

 Now let $\tree_2$ be the other tree of Figure~\ref{fig: trivalent for Gr24}. 
The shift map is
	\begin{align*}\widetilde{\mathsf{S}}_{12}:\widetilde{\Delta}_{\tree_1}&\to\widetilde{\Delta}_{\tree_2}\\
	(z_1,\ldots,z_{5})&\mapsto(z_1,\ldots,z_{4},z_{5}+\max(|z_1-z_4|,|z_2-z_3|)-\max(|z_1-z_2|,|z_3-z_4|)),
	\end{align*}
and the flip map is
	\begin{align*}\widetilde{\mathsf{F}}_{12}:\widetilde{\Delta}_{\tree_1}&\to\widetilde{\Delta}_{\tree_2}\\
	(z_1,\ldots,z_{5})&\mapsto(z_1,\ldots,z_{4},-z_{5}+\min(z_1+z_4,z_2+z_3)+\max(|z_1-z_2|,|z_3-z_4|).
	\end{align*}
Similarly, one can give an explicit description for the shift and flip maps for $\Delta_{\tree_1}\to\Delta_{\tree_2}$. 
\end{example}

\begin{remark} 
Our flip maps are related to cluster mutations in the case of $Gr(2,m)$. Recall that the homogeneous coordinate ring of the Grassmannian is a cluster algebra with the Pl\"ucker coordinates as its cluster variables \cite{FominZelevinsky} and this cluster structure gives rise to an atlas of complex tori on (an open dense subset inside) $Gr(2,m)$. The transition maps between adjacent tori are called \textit{(cluster) mutations}. In this case 
the tropicalized mutation coincides with our ``flip'' wall-crossing. 
Let us exhibit this in the example above. 
Starting with the seed $\{p_{12},p_{23},p_{34},p_{14},p_{13}\}$ and mutating at $p_{13}$ replaces this coordinate with
	$$
	\frac{p_{14}p_{23}+p_{12}p_{34}}{p_{13}}
	=
	p_{24}
	$$
yielding the seed $\{p_{12},p_{23},p_{34},p_{14},p_{24}\}$.
The tropicalization of this Laurent polynomial with respect to the maximum convention is $-p_{13}+\max(p_{14}+p_{23},p_{12}+p_{34})$.
By identifying the variables as follows
	$$
	z_1=p_{12},
	\qquad
	z_2=p_{23},
	\qquad
	z_3=p_{34},
	\qquad
	z_4=p_{14},
	\qquad
	z_5=p_{13}
	$$
	and using the identity 
	$$
	\min(z_a+z_d,z_b+z_c)
	+
	\max(|z_a-z_b|,|z_c-z_d|)
	=
	\max(z_a+z_b,z_c+z_d)
	$$
	we obtain $\widetilde{\mathsf{F}}_{12}$.
In \cite{Rietsch-Williams} Rietsch and Williams obtain piecewise linear maps for Newton-Okounkov bodies for $Gr(2,m)$, and more generally $Gr(k,m)$, by tropicalizing cluster mutation. See also \cite{BossHarMoha2019,BossMohaNaje} for related discussion. 
\end{remark}

\subsection{The algebraic wall-crossing map for $Gr(2,m)$}

In this section, we give a description of the algebraic wall-crossing map described in Section~\ref{subsec: algebraic crossing} for the case of $Gr(2,m)$, or more precisely  $\mathcal{T}(I_{2,m})$.  In the next section we will prove that the algebraic wall-crossing obtained below is also the restriction (to the semigroup) of the geometric ``flip'' map. 
For the definition of the algebraic wall-crossing, as explained in Section~\ref{subsec: algebraic crossing} we restrict to the case when the two adjacent maximal-dimensional prime cones of $\mathcal{T}(I_{2,m})$ lie in a certain maximal cone of the Gr\"obner fan. 
This may appear to be a restrictive condition. While not strictly logically necessary, we illustrate in the first few lemmas below that, up to the symmetry of $\mathfrak{S}_m$, this is true for any pair of adjacent maximal-dimensional prime cones of $\mathcal{T}(I_{2,m})$.

Recall that the semigroups
	$$S_{\tree_1}:=S(A, \nu_{M_{\tree_1}})
	\qquad \text{ and }\qquad 
	S_{\tree_2}:=S(A, \nu_{M_{\tree_2}})
	$$
are generated by the columns of $M_{\tree_1}$ and $M_{\tree_1}$, respectively.
As explained in Section~\ref{subsec: algebraic crossing}, we will use an adapted basis to construct the algebraic wall-crossing map $\Theta: S_{\tree_1}\to S_{\tree_2}$. To describe $\Theta$ more concretely we need some preliminaries. 
By \cite[\S 3.7]{SturmfelsInvariant} we know 
there exists a total order $\prec$ on $\C[p_{I}\mid I\subset [m], |I|=k]$ such that for any quadruple $\{i,j, k, \ell\}$ of indices in $[m]$ with $i<j<k<l$ we have that 
	\begin{equation}\label{eq:trop2nmonorder}
	\init_\prec(p_{ij}p_{kl}-p_{ik}p_{jl}+p_{il}p_{jk})=-p_{ik}p_{jl}.
	\end{equation}

The next lemma is essentially \cite[Second proof of $\supset$ in Theorem~4.3.5]{MaclaganSturmfels} and will be useful in what follows, so we briefly recall the idea of the argument. 
Note that the symmetric group $\mathfrak{S}_m$ naturally acts on the variables $p_{ij}$ by permuting the indices. 

\begin{lemma}
Up to this $\mathfrak{S}_m$ symmetry, all the maximal cones in $\mathcal{T}(I_{2,m})$ are contained in the maximal cone of the Gr\"obner fan of $I_{2,m}$ corresponding to the monomial order $\prec$ above. 
\end{lemma} 

\begin{proof} 
Let $\tree$ be a trivalent tree with $m$ leaves.
Fix a planar embedding of the graph where the $m$ leaves are arranged in a circle. 
We can act by $\mathfrak{S}_m$ to relabel the leaves so that they appear $1,2,\ldots, m$, in order, counterclockwise.  
We claim that, in this situation, the cone $C_\tree$ lies in the maximal cone of the Gr\"obner fan corresponding to the monomial order $\prec$ above. To see this, it suffices to check that for any choice of four leaves $1 \leq i < j < k < \ell \leq m$ of $\tree$, the initial term of the corresponding Pl\"ucker relation $\init_{C_{\tree}}(p_{ij}p_{k\ell}-p_{ik}p_{j\ell}+p_{i\ell}p_{jk})$ contains the monomial $-p_{ik}p_{j \ell}$, since this implies that $\prec$ refines the weight order corresponding to $C_\tree$. Recall that the cone $C_\tree$ is spanned by the tree distances of the form $\mathbb{1} - d_{\varepsilon}$ for interior edges $\varepsilon$, and also the lineality space. By definition, the lineality space does not affect the Pl\"ucker relations so it suffices to consider the interior edges. Let $\varepsilon$ be an interior edge. Due to the counterclockwise ordering of the vertices, it is not hard to see that if $\varepsilon$ has the property that $\lvert \{ i, j, k, \ell\} \rvert \cap J_\varepsilon \rvert = 2$, then we have
\begin{equation}\label{eq: Jepsilon} 
\{i, j, k, \ell \} \cap J_\varepsilon = \{a,b\} \quad \textup{ and } \quad \{i,j,k,\ell\} \cap J_\varepsilon^c = \{c,d\}
\end{equation}
where either $\{a,b\} = \{i,j\}$ or $\{a,b\} = \{i,\ell\}$. It can then be checked that
the initial term $\init_{\mathbb{1}-d_\varepsilon}(p_{ij}p_{k\ell}-p_{ik}p_{j\ell}+p_{i\ell}p_{jk})$ 
contains the monomial $-p_{ik}p_{j \ell}$ in either case.  For any other internal edge $\varepsilon'$, since $\tree$ is a tree we can see that either $\lvert \{ i, j, k, \ell\} \rvert \cap J_{\varepsilon'} \rvert \neq 2$ or, 
the decomposition $\{i,j,k,\ell\} = (\{i,j,k,\ell\} \cap J_\varepsilon) \sqcup (\{i,j,k,\ell\} \cap J_\varepsilon^c)$ is the same as that for $\varepsilon$ in~\eqref{eq: Jepsilon}. In the former case, 
\[
\init_{\mathbb{1} - d_{\varepsilon'}}(p_{ij}p_{k\ell}-p_{ik}p_{j\ell}+p_{i\ell}p_{jk}) = p_{ij}p_{k\ell}-p_{ik}p_{j\ell}+p_{i\ell}p_{jk}
\]
and in the latter case, the initial term is the same as that for $\varepsilon$.
It follows that $\init_{C_\tree}(p_{ij}p_{k\ell}-p_{ik}p_{j\ell}+p_{i\ell}p_{jk}) =  - p_{ac}p_{bd} + p_{ad}p_{bc}$, as desired. 
\end{proof}

	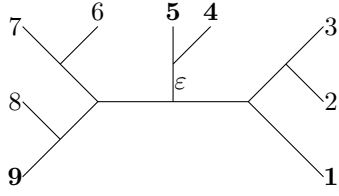
\begin{figure}[h] 
	\begin{tikzpicture}
	\draw (1,1)--(1,0);
	\draw (1,.5)--(1.5,1);
	\draw (0,0)--(2,0);
	\draw (3,1)--(2,0)--(3,-1);
	\draw (2.5,.5)--(3,0);
	\draw (-1,1)--(0,0)--(-1,-1);
	\draw (-.5,.5)--(0,1);
	\draw (-1,0)--(-.5,-.5);
	\node at (-1.1,-1) {$\boldsymbol{9}$};
	\node at (-1.1,0) {$8$};	
	\node at (-1.1,1) {$7$};
	\node at (0,1.2) {$6$};
	\node at (1,1.2) {$\boldsymbol{5}$};	
	\node at (1.5,1.2) {$\boldsymbol{4}$};	
	\node at (3.1,1) {$3$};
	\node at (3.1,0) {$2$};
	\node at (1.1,.25) {$\varepsilon$};	
	\node at (3.1,-1) {$\boldsymbol{1}$};
	\end{tikzpicture}
	\caption{The initial form of the Pl\"ucker relation for $1,4,5,9$ with respect to the tree above is $p_{14}p_{59}-p_{15}p_{49}$.}
	\label{fig: initial form of tree} 
	\end{figure}

To describe the algebraic wall-crossing map, we first need to act by $\mathfrak{S}_n$ to simultaneously take two adjacent maximal-dimensional prime cones in $\mathcal{T}(I_{2,n})$ to the same maximal cone in the Gr\"obner fan. This is the content of the next lemma. 

\begin{lemma} 
Let $\tree_1$ and $\tree_2$ correspond to two adjacent maximal cones in $\mathcal{T}(I_{2,n})$. Then there exists an element of $\mathfrak{S}_n$ which takes both $C_{\tree_1}$ and $C_{\tree_2}$ to the maximal cone in the Gr\"obner fan corresponding to $\prec$. 
\end{lemma} 

\begin{proof} 
Let $\varepsilon_{2m-3} \in\tree_1$ and $\varepsilon'_{2m-3} \in\tree_2$ be the edges by which the two trees differ, as in Figure~\ref{fig: tree flip}. 
There exist $I,J,K,L\subset [m]$ such that $\varepsilon_1$ corresponds to the partition $I\cup J, K\cup L$ and $\varepsilon_2$ corresponds to the partition $I\cup L, J\cup K$.
There are planar realizations of $\tree_1$ and $\tree_2$ with the leaves arranged in a circle so that $I,J,K,L$ are arranged in counterclockwise order.
\end{proof} 

The above discussion shows that the assumption that two adjacent maximal-dimensional prime cones are both contained in the same maximal cone of a Gr\"obner fan is not very restrictive. 
With this in mind, we now describe the wall-crossing $\Theta$ for value semigroups, under the assumption that both $C_{\tau_1}$ and $C_{\tau_2}$ lie in the maximal Gr\"obner cone corresponding to the above monomial order $\prec$. 
By Theorem~\ref{theorem:kavehmanon theorem 4} (2), we may take 
the projection onto $\C[p_J : J \subset [m], \lvert J \rvert =2]/I_{2,m}$ of the standard monomial basis for $I_{2,m}$ with respect to $\prec$ as an adapted basis $\mathbb{B}$ for $(A,\nu_{M_{\tau_k}})$. 
Since the Pl\"ucker relations are a Gr\"obner basis for $\prec$ \cite[Second proof of $\supseteq$ of Theorem 4.3.5]{MaclaganSturmfels} we have 
	$$\init_\prec(I_{2,m})=\langle p_{ik}p_{jl} \mid 1\le i<j<k<l\le m \rangle$$
then $\mathbb{B}$ is the image of $\mathcal{S}(\prec,I_{2,m})$ under the projection $\pi:\C[p_{I}]\to\C[p_{I}]/I_{2,m}$, where
\begin{equation}\label{eq: standard condition} 
	\mathcal{S}(\prec,I_{2,m})=\left\{\prod_{i<j}p_{ij}^{\alpha_{ij}} \mid \alpha_{ik}\alpha_{jl}=0\quad \text{for}\quad 1\le i<j<k<l\le m\right\}.
	\end{equation} 
As before, the maps 
	\begin{align*}
	\theta_k:\mathbb{B} & \to S(A, \nu_{M_{\tree_k}})\\
	b & \mapsto  \nu_{M_{\tree_k}}(b)
	\end{align*}
are bijections for $k=1,2$. 
We then obtain the bijection $\Theta:S_{\tree_1}\to S_{\tree_2}$ defined by $\Theta=\theta_2\circ\theta_1^{-1}$ as in Section~\ref{subsec: algebraic crossing}. We work out a concrete example below.

\begin{example}\label{ex: 2,4 alg map} 
Let $m=4$ and consider $\mathcal{T}(I_{2,4})$. Let $\tau_1$ and $\tau_2$ be as in Example~\ref{ex: ineqs and gemo maps} and let 
$e_{12},\ldots,e_{34}$ (respectively $f_{12}, \ldots, f_{34}$) denote the columns of $M_{\tree_1}$ (respectively $M_{\tree_2}$), labelled by the same Pl\"ucker indices as the columns themselves. We claim that the algebraic map is given by the concrete formula 
	\begin{equation}\label{eq: 2,4 alg map}
	\Theta(M_{\tree_1}\alpha)
	=
	M_{\tau_2}\begin{pmatrix}\alpha_{12}\\ \max(0,\alpha_{13}-\alpha_{24}) \\ \alpha_{14}+\min(\alpha_{13},\alpha_{24})\\ \alpha_{23}+\min(\alpha_{13},\alpha_{24}) \\ \max(0,\alpha_{24}-\alpha_{13})\\ \alpha_{34}\end{pmatrix}.
	\end{equation}
To see this, first let $M_{\tree_1}\alpha\in S(A, \nu_{M_{\tree_1}})$ for an arbitrary $\alpha \in \Z^{\binom{4}{2}}_{\geq 0}$. We check~\eqref{eq: 2,4 alg map} by cases.  Note that the only set of indices satisfying the condition $1 \leq i < j < k < \ell \leq 4$ is $i=1, j=2, k=3, \ell=4$. Hence if $\alpha_{13}\alpha_{24}=0$, then by~\eqref{eq: standard condition} it follows that $p^\alpha\in \mathcal{S}(\prec,I_{2,m})$ and therefore $\Theta(M_{\tree_1}\alpha)=M_{\tree_2}\alpha$, which agrees with \eqref{eq: 2,4 alg map}. On the other hand, 
if $\alpha_{13}\alpha_{24}\neq 0$, then from the definition of $\Theta$ we must first find $\beta \in \Z^{\binom{4}{2}}_{\geq 0}$ such that $p^{\beta} \in \mathcal{S}(\prec,I_{2,m})$ and such that $\nu_{M_{\tree_1}}(\pi(p^\alpha))=\nu_{M_{\tree_1}}(\pi(p^\beta))$. 
This would imply that $\Theta(M_{\tree_1}\alpha)=M_{\tree_2}\beta$. 
To achieve this, we can use the relation $p_{13}p_{24}=p_{14}p_{23}$ in $\C[p_I]/\init_{C_{\tree_1}}(I_{2,4})$ to see that
	$$
	p_{13}^{\alpha_{13}}p_{24}^{\alpha_{24}}
	=
	\begin{cases}
	p_{14}^{\alpha_{13}}p_{23}^{\alpha_{13}}p_{24}^{\alpha_{24}-\alpha_{13}}\quad \text{ if }\quad \alpha_{24}\ge\alpha_{13}\\
	p_{13}^{\alpha_{13}-\alpha_{24}}p_{14}^{\alpha_{13}}p_{23}^{\alpha_{13}}\quad \text{ if }\quad \alpha_{13}\ge\alpha_{24}
	\end{cases}
	$$
in $\C[p_I]/\init_{C_{\tree_1}}(I_{2,4})$. The vector $\beta$ can be found by using the above substitution in $p^\alpha$. Then~\eqref{eq: 2,4 alg map} follows from combining the cases.
\end{example}

In general, the algebraic wall-crossing map is difficult to describe explicitly.

\subsection{The geometric wall-crossing map $\mathsf{F}_{12}$ induces the algebraic wall-crossing map for $Gr(2,m)$}\label{subsec: algebraic is flip for Gr2m}

The main result of this section is the following.

\begin{theorem}\label{theorem: main Gr2m}
\textit{ Let $I_{2,m}$ be the Pl\"ucker ideal for $Gr(2,m)$ and let $C_{\tree_1}$ and $C_{\tree_2}$ be two maximal-dimensional prime cones in $\mathcal{T}(I_{2,m})$ which share a codimension-$1$ face. Assume that $C_{\tree_1}$ and $C_{\tree_2}$ are contained in the maximal cone of the Gr\"obner fan corresponding to $\prec$. 
Extend the flip geometric wall-crossing map $\mathsf{F}_{12}$ of~\eqref{eq: flip map} to a map $\mathsf{F}_{12}:P_{\tree_1}\rightarrow P_{\tree_2}$ where $P_{\tree_i}$ is the cone spanned by the columns of $M_{\tree_i}$.
Then the algebraic wall-crossing map $\Theta:S_{\tree_1}\rightarrow S_{\tree_2}$ is the restriction of the flip geometric wall-crossing map $\mathsf{F}_{12}$.}
\end{theorem}

Our method of proof is to first show the analogous result for the matrices $\widetilde{M}_{\tree_i}$ introduced in Section~\ref{subsec: geometric Gr2m}. Since the $\widetilde{M}_{\tree_i}$ and $M_{\tree_i}$ are related by the change of coordinates $\gamma$ of~\eqref{eq: gamma}, this gives us the desired result.

First, we define a map 
	\begin{equation}\label{eq: def widetilde Theta}
	\widetilde{\Theta}:= \gamma^{-1} \circ \Theta \circ \gamma. 
	\end{equation}
	By Lemma~\ref{lemma: gamma} we know that $\gamma$ induces a bijection from $\widetilde{S}_{\tree_i}$ to $S_{\tree_i}$, so it follows immediately that $\widetilde{\Theta}$ restricts to a map 
	\[
	\widetilde{\Theta}: \widetilde{S}_{\tau_1} \to \widetilde{S}_{\tau_2}.
	\]
Since $\gamma$ takes columns of $\widetilde{M}_{\tree_i}$ to the corresponding columns of $M_{\tree_i}$ for $i=1,2$, it follows easily that the map $\widetilde{\Theta}$ of~\eqref{eq: def widetilde Theta} behaves as follows: 
\[
\widetilde{\Theta}(\widetilde{M}_{\tree_1}\alpha) = \widetilde{M}_{\tree_2}\alpha
\]
for $\alpha$ such that $p^{\alpha} \in \mathcal{S}(\prec, I_{2,m})$. 	

The following proposition is the analogue of Theorem~\ref{theorem: main Gr2m} for the map $\widetilde{\Theta}$. 
Let $\widetilde{\sf F}_{12}$ be the flip map defined in Section~\ref{subsec: geometric crossing Gr2m}. 

\begin{proposition}\label{lem: cal gr2n geom equals alg} 
The restriction of $\widetilde{{\sf F}}_{12}$ to the semigroup $\widetilde{S}_1$ is equal to $\widetilde{\Theta}: \tilde{S}_1 \to \tilde{S}_2$, i.e., for any $\alpha$ such that $p^{\alpha} \in \mathcal{S}(\prec,I_{2,m})$ we have 
	$$\widetilde{{\sf F}}_{12}(\widetilde{M}_{\tree_1}\alpha)
	=\widetilde{\Theta}(\widetilde{M}_{\tree_1}\alpha).
	$$
\end{proposition}

The following lemma will be useful to prove the above proposition. Here we use the notation introduced in Figure~\ref{fig: tree flip}.

\begin{lemma}\label{lem: last entry flip}
Let $\tree_1$ and $\tree_2$ correspond to two adjacent maximal-dimensional prime cones $C_{\tree_1}$ and $C_{\tree_2}$ in 
$\mathcal{T}(I_{2,m})$. Assume that $C_{\tree_1}$ and $C_{\tree_2}$ both lie in the maximal cone of the Gr\"obner fan corresponding to $\prec$. 
For $\alpha \in \Z^{\binom{m}{2}}_{\geq 0}$,
	$$
	(\widetilde{{\sf F}}_{12}(\widetilde{M}_{\tree_1}\alpha))_{2m-3}= \alpha_{IJ}+|\alpha_{IK}-\alpha_{JL}|+\alpha_{KL}
	$$
where the $\alpha_{IJ},\alpha_{IK},\ldots$ denote the sums of the form
	$$
	\alpha_{IJ}:=\sum_{i\in I, j\in J}\alpha_{ij}
	$$
	and similarly for the others. 
\end{lemma}

\begin{proof} 
	Recall from \eqref{eq: tilde flip} that if $(z_a)_{a=1,\ldots,2m-3}$ denote the coordinate entries of $\widetilde{M}_{\tree_1}\alpha$, then
\begin{equation}\label{eq: F12 formula}
	(\widetilde{{\sf F}}_{12}(\widetilde{M}_{\tree_1}\alpha))_{2m-3}
	=
	-z_{2m-3}+\min(z_a+z_d,z_b+z_c)+\max(|z_a-z_b|,|z_c-z_d|).
	\end{equation} 
From the formula~\eqref{eq: cijk def} for the matrix entries of $\widetilde{M}_{\tree_1}$, it follows -- using the notation of Figure~\ref{fig: tree flip} -- that 
for any edge $\varepsilon_h$, the coordinate $z_h$ in $\widetilde{M}_{\tree_1} \alpha$ equals the sum of the exponents $\alpha_{ij}$ such that the path from $i$ to $j$ uses $\varepsilon_h$.
Applying this to $a$, $b$, $c$, $d$, and $2m-3$ we conclude
	\begin{align*}
	z_{a}&=\sum\alpha_{ij}c^{ij}_{a}=\alpha_{IJ}+\alpha_{IK}+\alpha_{IL},
	&z_{d}=\sum\alpha_{ij}c^{ij}_{d}=\alpha_{IL}+\alpha_{JL}+\alpha_{KL},\\
	z_{b}&=\sum\alpha_{ij}c^{ij}_{b}=\alpha_{IJ}+\alpha_{JK}+\alpha_{JL},
	&z_{c}=\sum\alpha_{ij}c^{ij}_{c}=\alpha_{IK}+\alpha_{JK}+\alpha_{KL}.
	\end{align*} 
	
	We also have 
	\begin{equation}\label{eq: formula z2m-3}
	 z_{2m-3}=\sum\alpha_{ij}c^{ij}_{2m-3}=\alpha_{IK}+\alpha_{IL}+\alpha_{JK}+\alpha_{JL}.
	\end{equation} 
	
It is straightforward to compute
	\begin{equation}\label{eq: min}
	\min(z_a+z_d,z_b+z_c)
	=
	\alpha_{IJ}+\alpha_{IK}+\alpha_{JL}+\alpha_{KL}+2\min(\alpha_{IL},\alpha_{JK}),
	\end{equation}
and
	\begin{equation}\label{eq: max} 
	\max(|z_a-z_b|,|z_c-z_d|)
	= 	|\alpha_{JK}-\alpha_{IL}|+|\alpha_{JL}-\alpha_{IK}|.
	\end{equation}
Since 
	$$
	2\min(\alpha_{IL},\alpha_{JK})+|\alpha_{JK}-\alpha_{IL}|=\alpha_{IL}+\alpha_{JK},
	$$
then by combining~\eqref{eq: F12 formula}, ~\eqref{eq: formula z2m-3}, ~\eqref{eq: min} and~\eqref{eq: max} we obtain 
	\begin{equation*}
	(\widetilde{{\sf F}}_{12}(\widetilde{M}_{\tree_1}\alpha))_{2m-3}
	=\alpha_{IJ}+|\alpha_{JL}-\alpha_{IK}|+\alpha_{KL} 
	\end{equation*} 
	as desired. 
\end{proof}

We now prove Proposition~\ref{lem: cal gr2n geom equals alg}.

\begin{proof}[Proof of Proposition~\ref{lem: cal gr2n geom equals alg}]
We first claim that $\widetilde{\Theta}$ and $\widetilde{{\sf F}}_{12}$ agree on the first $2m-2$ coordinates. Indeed, note that the flip map ${\sf F}_{12}$ is the identity on the first $2m-2$ coordinates and thus so is $\widetilde{{\sf F}}_{12} = \gamma^{-1} \circ {\sf F}_{12} \circ \gamma$. The same is true of $\widetilde{\Theta}$, since the matrices $\widetilde{M}_{\tree_1}$ and $\widetilde{M}_{\tree_2}$ are identical except for the bottom rows. Thus it remains to see that $\widetilde{\Theta}$ and $\widetilde{{\sf F}}_{12}$ agree on the last $(2m-3)$-th coordinate.

Let $\alpha\in\Z_{\ge 0}^{\binom{m}{2}}$ such that $p^{\alpha} \in \mathcal{S}(\prec,I_{2,m})$. 
In the notation of Figure~\ref{fig: tree flip}, where the right hand figure corresponds to the tree $\tree_2$ and from the formula~\eqref{eq: cijk def} for the entries of $\widetilde{M}_{\tree_2}$ we conclude that 
the $2m-3$-th coordinate of $\widetilde{\Theta}(\widetilde{M}_{\tree_1}\alpha) = \widetilde{M}_{\tree_2}\alpha$ is
	\begin{equation}
	(\widetilde{M}_{\tree_2}\alpha)_{2m-3}=
	\alpha_{IJ}+\alpha_{IK}+\alpha_{JL}+\alpha_{KL}.
	\label{eq: no cross algebraic}\end{equation}
We take cases. If $\alpha_{IK}\neq 0$, then since $\alpha$ corresponds to a standard monomial we know $\alpha_{JL}=0$. Therefore by \eqref{eq: no cross algebraic} and Lemma~\ref{lem: last entry flip} we conclude 
	$$
	(\widetilde{{\sf F}}_{12}(\widetilde{M}_{\tree_1}\alpha))_{2m-3}
	=\alpha_{IJ}+\alpha_{IK}+\alpha_{KL}
	=(\widetilde{\Theta}(\widetilde{M}_{\tree_1}\alpha))_{2m-3}.
	$$
On the other hand, $\alpha_{IK}=0$ then by \eqref{eq: no cross algebraic} and Lemma~\ref{lem: last entry flip} we have
	$$
	(\widetilde{{\sf F}}_{12}(\widetilde{M}_{\tree_1}\alpha))_{2m-3}
	=\alpha_{IK}+\alpha_{JK}+\alpha_{KL}
	=(\widetilde{\Theta}(\widetilde{M}_{\tree_1}\alpha))_{2m-3}
	$$
	which proves the claim. 
\end{proof}

\begin{proof}[Proof of Theorem~\ref{theorem: main Gr2m}]
By Remark~\ref{rem: maps extend to cones} the geometric wall-crossing map $\mathsf{F}_{12}$ of~\eqref{eq: flip map} is a wall-crossing map for the cones which restricts to the flip geometric wall-crossing of the Newton-Okounkov bodies.
Let $M_{\tree_1}\alpha\in S_{\tree_1}$. Since $\gamma$ is an invertible linear map mapping, for all $ij$, the $ij$-th column of $\widetilde{M}_{\tree_1}$ to the $ij$-th column of $M_{\tree_2}$ then 
	\begin{equation}
	\gamma(\widetilde{M}_{\tree_1}\alpha)=M_{\tree_1}\alpha \quad \text{and} \quad \gamma^{-1}(M_{\tree_1}\alpha)=\widetilde{M}_{\tree_1}\alpha.
	\label{eq: gamma is linear}\end{equation}
Let $\alpha$ be such that $p^\alpha\in \mathcal{S}(\prec,I_{2,m})$. We compute 
	\begin{align*}
	{\sf F}_{12}(M_{\tree_1}\alpha)&=\gamma\circ \widetilde{{\sf F}}_{12}\circ\gamma^{-1}(M_{\tree_1}\alpha),  &\text{by~\ref{eq: flip to flip}}\\
	&=\gamma\circ \widetilde{{\sf F}}_{12}(\widetilde{M}_{\tree_1}\alpha), &\text{by \eqref{eq: gamma is linear}}\\
	&=\gamma(\widetilde{\Theta}(\widetilde{M}_{\tree_1} \alpha)) & \text{ by Proposition~\ref{lem: cal gr2n geom equals alg}} \\ 
	& = \gamma(\widetilde{M}_{\tree_2} \alpha) & \textup{ by definition of $\widetilde{\Theta}$} \\ 
	& = M_{\tree_1}\alpha & \textup{ by ~\eqref{eq: gamma is linear}} \\
	& = \Theta(M_{\tree_1}\alpha) & \textup{ by definition of $\widetilde{\Theta}$} 
	\end{align*}
	as desired. 
\end{proof}

\appendix
\section{ Newton-Okounkov Body Wall-crossing via Complexity-One $T$-Varieties}
\begin{center}
	\emph{by} {\bf Nathan Ilten}\footnote{
	Department of Mathematics, Simon Fraser University, 8888 University Drive, Burnaby, British Columbia, V5A 1S6, Canada. Email: \texttt{nilten@sfu.ca
	}
	}

\end{center}

\subsection{Preliminaries}\label{sec:prelim}
We continue with notation as in the main paper. In particular, we have a presentation $\C[x_1,\ldots,x_n]/I\cong A$ of a $(d+1)$-dimensional positively graded integral domain, and two $(d+1)$-dimensional prime cones $C_1$ and $C_2$ in $\trop(I)$ which intersect in a codimension-one face $C$. 
In this appendix, we show how the main features of wall-crossing for the Newton-Okounkov bodies as outlined in Theorem~\ref{theorem:main} follow from the theory of complexity-one $T$-varieties.

The main algebraic objects we will consider are as follows. 
Let $N=\Z^n\cap \langle C \rangle$, and $R=\C[x_1,\ldots,x_n]/\init_C(I)$.
The dual lattice to $N$ is  $M=(\Z^n)^*/N^\perp =N^*$.
Then
\begin{enumerate}
	\item The ring $R$ is an $M$-graded integral domain of dimension $d+1$;
	\item $R$ is finitely generated as a $\C$-algebra;
	\item The degree zero piece $R_0$ of $R$ is $\C$;
	\item The set of those $v\in M$ with $R_v\neq 0$ generates all of $M$, which is a rank $d$ lattice.
\end{enumerate}
Rings $R$ satisfying these four properties are exactly the coordinate rings of (potentially non-normal) affine complexity-one $T$-varieties with a good $T$-action (where $T$ is the algebraic torus $\Spec \C[M]$).\footnote{Recall that a torus action on an affine variety is \emph{good} if the only invariant regular functions are constants.} We will thus call rings satisfying these four properties \emph{good complexity-one} $M$-graded domains.

\subsection{Polyhedral Divisors}\label{sec:pdiv}
We fix a lattice $M$ and a smooth projective curve $Y$. 
Let $\omega$ be a full-dimensional cone in $M_\R=M\otimes \R$.
	A \emph{polyhedral divisor} on $Y$ with weight cone $\omega$ is a 
	finite formal sum 
	\[
		\D=\sum_{P\in\mcP} \D_P\otimes P
	\]
	where $\mcP$ is a finite set of points of $Y$ and the $\D_P$ are piecewise linear concave functions 
	\[
\D_P:\omega\to \R
	\]
with rational slopes. See {\cite[\S2-3]{AH}} for details.
For any $v\in \omega\cap M$, we obtain a $\Q$-divisor \[\D(v):=\sum_{P\in\mcP} \D_P(v)\cdot P\] on $Y$.
	We say $\D$ is a \emph{p-divisor} if  $\deg \D(v)>0$ for $v\in M$ in the interior of $\omega$, and for every $v\in M$ in the boundary of $\omega$, either $\deg \D(v)>0$, or $\D(v)$ has a principal multiple.

To any p-divisor $\D$ as above, we may associate a \emph{normal} 
good complexity-one $M$-graded domain \cite[Theorem 3.1]{AH}:
\[
R(\D)=	\bigoplus_{v\in \omega \cap M} H^0\left(Y,\CO_Y\left(\D(v)\right)\right)\cdot \chi^{v}.
\]
Furthermore, every normal good complexity-one $M$-graded domain $R$ arises in this fashion \cite[Theorem 3.4]{AH}.
In geometric terms, there is a bijection between equivariant isomorphism classes of normal affine varieties with good complexity-one torus action, and p-divisors on smooth projective curves modulo a natural equivalence relation, see \cite{AH}.

\subsection{Newton-Okounkov Bodies}\label{sec:NO}
Consider a good complexity-one $M$-graded domain $R$.
We now fix a $\Z$-grading on $R$ by considering a projection $\deg: M\to \Z$ satisfying $\deg^{-1}(0)\cap \omega=0$. Set $\Box=\omega\cap \deg^{-1}(1)$.
By the discussion of \S\ref{sec:pdiv}, there is a p-divisor $\D=\sum_{P\in\mcP} \D_P\otimes P$ on a curve $Y$ such that the integral closure of $R$ is isomorphic to $R(\D)$; we identify $R$ with its image in $R(\D)$. In general, $R$ may not be equal to $R(\D)$, as $R$ may not be integrally closed.

Fix a total ordering on $M$. For any point $Q\in Y$, we obtain a valuation
\begin{align*}
\val_Q:R(\D)\setminus \{0\}&\to M\times \Z\\
f&\mapsto \min_{v: f_v\neq 0} (v,\ord_Q(f_v))
\end{align*}
where $f=\sum_{v\in M} f_v$ is the decomposition of $f$ into homogeneous pieces, $\ord_Q(f)$ is the order of vanishing of $f$ at $Q$, and we take the lexicographic ordering on $M\times \Z$.
The valuation $\val_Q$ restricts to a valuation on $R$.

\begin{lemma}\label{lemma:NO}
The Newton-Okounkov body $\Delta(R,\val_Q)$ is equal to
\begin{equation}\label{eqn:NO}
\tag{$\dagger$}	\Delta(R,\val_Q)= \left\{ (x,y)\in\Box\times \R\ |\ -\D_Q(x)\leq y \leq \sum_{\substack{P\in\mcP\\P\neq  Q}} \D_P(x)\right\}.
\end{equation}
For the case $Q\notin \mcP$, we use the convention that $\D_Q=0$.
\end{lemma}
\begin{proof}
	We observe that $\Delta(R,\val_Q)=\Delta(R(\D),\val_Q)$ by e.g.~\cite[Proposition 2.18]{IW}, so we reduce to the case $R=R(\D)$. In the case $Y=\P^1$, we may now apply \cite[Theorem 5.10]{IM}. For arbitrary $Y$,  we may apply Petersen's description of Newton-Okounkov bodies for complexity-one $T$-varieties \cite[Proposition 3.13]{Pet}. 
	For the sake of the reader, we reproduce Petersen's argument below.

	Consider any homogeneous element $f\cdot \chi^v\in R(\D)$ of degree $v$. Since $f$ is in $H^0(Y,\CO(\D(v)))$, we obtain $0\leq \ord_Q f+\D_Q(v)\leq \sum_{P\in\mcP} \D_P(v)$. Since the image of $\val_Q$ is determined by valuations of homogeneous elements, it follows that $\Delta(R,\val_Q)$ is contained in the expression on the right hand side of \eqref{eqn:NO}.

	Conversely, choose any rational point $x$ in the relative interior of $\Box$ and $y\in \Q$ with $(x,y)$ contained in the right hand side of \eqref{eqn:NO}. There exists a natural number $\lambda$ such that $v=\lambda\cdot x\in M$ and $\D_P(v)$ is integral for all $P\in\mcP$. Since $x$ is in the relative interior of $\Box$, the $\Z$-divisor $\D(v)$ has positive degree. By the theorem of Riemann-Roch, there thus exists a sequence of sections $s_i\in H^0(Y,\CO(\D(i\cdot v)))$  such that 
	\[
		\lim_{i\to \infty} \frac{\ord_Q(s_i)}{i\cdot \lambda}=y.
	\]
Since $\Delta(R,\val_Q)$ is a closed set, this implies that $(x,y)$, and thus the entire right hand side of \eqref{eqn:NO}, is contained in $\Delta(R,\val_Q)$.
\end{proof}

\subsection{Wall-Crossing}\label{sec:wall}
We now take $R$ to be as in \S\ref{sec:prelim}, the degeneration of $A$ corresponding to the cone $C=C_1\cap C_2$. Let $u_1,\ldots,u_d\in \Z^{n}$ be elements of the relative interior of $C$ which form a lattice basis for $N$. We assume that $u_1\in\Z^n$ is the primitive vector giving the $\Z$-grading of the variables $x_i$; in the standard graded case it is just $(1,\ldots,1)$. Likewise, for $i=1,2$ let $w_i\in \Z^{n}$ be in the relative interior of $C_i$ such that $u_1,\ldots,u_d,w_i$ form a lattice basis for $\langle C_i\rangle \cap \Z^{n}$. As in \S\ref{section:background}, the collection of vectors $u_1,\ldots,u_d,w_i$ gives rise to valuations on both $A$ and $R$, both of which we denote by $\val_i$. 

\begin{theorem}\label{thm:NO}
	There exists a rational polytope $\Delta\subset \{1\}\times \R^{d-1}\subset \R^d$, and piecewise affine-linear concave functions $\Psi_0,\Psi_1,\Psi_2$ with rational slopes and translation from $\Delta$ to $\R$ satisfying $\Psi_0+\Psi_1+\Psi_2\geq 0$, such that 
\begin{align*}
&	\Delta(A,\val_1)= \left\{ (x,y)\in\Delta\times \R\ |\ -\Psi_1(x)\leq y \leq \Psi_2(x)+\Psi_0(x)\right\};\\
&	\Delta(A,\val_2)= \left\{ (x,y)\in\Delta\times \R\ |\ -\Psi_2(x)\leq y \leq \Psi_1(x)+\Psi_0(x)\right\}.
\end{align*}
In particular, under the projection from $\R^d\times \R$ to $\R^d$,
$\Delta(A,\val_1)$ and $\Delta(A,\val_2)$ have the same image $\Delta$, with fibers of equal Euclidean lengths.
\end{theorem}
\begin{proof}
By construction, $\Delta(A,\val_i)=\Delta(R,\val_i)$; we will show that
$\Delta(R,\val_i)$ has the desired form.
As in \S\ref{sec:NO}, we identify $R$ with a subring of $R(\D)$ for some p-divisor $\D$ on a curve $Y$. The vectors $u_1,\ldots,u_d$ give an isomorphism 
$\phi:M\to \Z^d$ by sending $v\in M$ to $(\langle u_1,v\rangle,\ldots,\langle u_d,v\rangle)$.
A straightforward adaptation of the arguments of \cite[Proposition 5.1]{IM} from the case $Y=\P^1$ to arbitrary $Y$ shows that 
there exists a point $Q_i\in Y$, constant $c_i\in \N$, and linear map $\gamma_i:M\to \Z$ such that 
the valuation $\val_i:R\setminus\{0\}\to \Z^d\times \Z$ has the form
\[
	\val_i(f)=\min_{v:f_v\neq 0} (\phi(v),c_i\cdot \ord_{Q_i}(f_v)+\gamma_i(v))
\]
for $f=\sum_{v\in M} f_v$.
Since $u_1,\ldots,u_d,w_i$ is a basis for $\langle C_i \rangle \cap \Z^n$, the group generated by $\val_i(R\setminus\{0\})$ must be all of $\Z^d\times \Z^1$. This implies that $c_i=1$.

Both valuations $\val_i$ and $\val_{Q_i}$ are determined by their behavior on $M$-homogeneous elements. Together with Lemma \ref{lemma:NO}, this implies that 
	$\Delta(R,\val_i)$ consists of those $(x,y)\in\phi(\Box)\times \R$ satisfying
	\[\gamma_i(\phi^{-1}(x))-\D_{Q_i}(\phi^{-1}(x))\leq y \leq\gamma_i(\phi^{-1}(x))+ \sum_{\substack{P\in\mcP\\P\neq Q_i}} \D_P(\phi^{-1}(x)).
\]
By abuse of notation, we are using $\phi$ and $\gamma_i$ to also denote their linear extensions to $M_\R$.
As in Lemma \ref{lemma:NO}, $\D_{Q}=0$ for $Q\notin \mcP$.
The map $\deg:M\to \Z$ giving a $\Z$-grading on $R$ is the map induced by $u_1\in N$.

We set 
\begin{align*}
	\Delta=\phi(\Box),\quad
	\Psi_0&=\gamma_1\circ\phi^{-1}+\gamma_2\circ\phi^{-1}+\sum_{\substack{P\in\mcP\\P\neq Q_1,Q_2}} \D_P\circ\phi^{-1};\\
	\Psi_i&=\D_{Q_i}\circ\phi^{-1}-\gamma_i\circ \phi^{-1}\qquad{i=1,2}
\end{align*}
to obtain the main claim of the theorem. For the claim regarding fiber lengths, we observe that for both Newton-Okounkov bodies, the length of the fiber over $x\in \Delta$ is exactly $\Psi_0(x)+\Psi_1(x)+\Psi_2(x)$. 
\end{proof}
\subsection{Example}\label{sec:ex}
Consider the ideal $I=\langle x_1x_2-x_3x_4-x_4^2-x_5^2\rangle\subset\C[x_1,\ldots,x_5]$. 
The tropicalization $\trop(I)$ is the product of its $2$-dimensional lineality space with a cone over the complete graph $K_4$ on four vertices.
The initial ideals $I_1=\langle x_1x_2-x_3x_4\rangle$ and $I_2=\langle x_1x_2-x_4^2\rangle$ correspond to prime cones $C_1$ and $C_2$;  for $C=C_1\cap C_2$ we have $\init_C(I)=\langle x_1x_2-x_3x_4-x_4^2\rangle$.

We may take the elements $u_1,u_2,u_3,w_1,w_2$ to be the rows of the matrix
{\scriptsize{\[
\left(
\begin{array}{c c c c c}
1&1&1&1&1\\
1&-1&0&0&0\\
0&0&0&0&1\\
\hline
1&0&0&1&1\\
\hline
0&0&1&0&0
\end{array}\right).
\]}}
\noindent Then $\Delta(A,\val_1)$ and $\Delta(A,\val_2)$ are the convex hulls of the columns of this matrix, after removing the fifth and fourth rows, respectively.

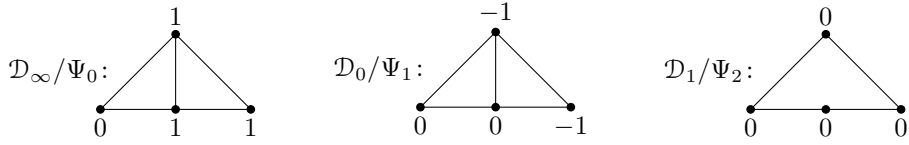
\begin{figure}
\begin{center}
\begin{tikzpicture}
\draw (0,0) -- (2,0) -- (1,1) -- (0,0);
\draw (1,0) -- (1,1);
\draw[fill] (0,0) circle [radius=.05];
\draw[fill] (1,0) circle [radius=.05];
\draw[fill] (2,0) circle [radius=.05];
\draw[fill] (1,1) circle [radius=.05];
\node [below] at (0,0) {$0$};
\node [below] at (1,0) {$1$};
\node [below] at (2,0) {$1$};
\node [above] at (1,1) {$1$};
\node  at (-.5,.5) {$\D_\infty/\Psi_0\colon$};
\end{tikzpicture}
\hspace{.6cm}
\begin{tikzpicture}
\draw (0,0) -- (2,0) -- (1,1) -- (0,0);
\draw (1,0) -- (1,1);
\draw[fill] (0,0) circle [radius=.05];
\draw[fill] (1,0) circle [radius=.05];
\draw[fill] (2,0) circle [radius=.05];
\draw[fill] (1,1) circle [radius=.05];
\node [below] at (0,0) {$0$};
\node [below] at (1,0) {$0$};
\node [below] at (2,0) {$-1$};
\node [above] at (1,1) {$-1$};
\node  at (-.5,.5) {$\D_0/\Psi_1\colon$};
\end{tikzpicture}
\hspace{.6cm}
\begin{tikzpicture}
\draw (0,0) -- (2,0) -- (1,1) -- (0,0);
\draw[fill] (0,0) circle [radius=.05];
\draw[fill] (1,0) circle [radius=.05];
\draw[fill] (2,0) circle [radius=.05];
\draw[fill] (1,1) circle [radius=.05];
\node [below] at (0,0) {$0$};
\node [below] at (1,0) {$0$};
\node [below] at (2,0) {$0$};
\node [above] at (1,1) {$0$};
\node  at (-.5,.5) {$\D_1/\Psi_2\colon$};
\end{tikzpicture}
\end{center}
\caption{Piecewise linear functions for wall-crossing}\label{fig:D}
\end{figure}

We now view this example from the perspective of $T$-varieties. The lattice $N$ is the rank three lattice generated by the first three rows of the above matrix. This gives an identification of $N$ with $\Z^3$, and we obtain an induced identification of $M$ with $\Z^3$. Under this identification, the image of $x_i$ in $R=\C[x_1,\ldots,x_5]/\init_C(I)$ is homogeneous of degree equal to the first three entries in the $i$th column of the above matrix. 

Let $\omega$ be the cone in $\R^3$ generated by $(1,1,0)$, $(1,-1,0)$, and $(1,0,1)$. The convex hull $\Delta$ of these three vectors is the affine slice of $\omega$ on which the first coordinate is equal to $1$. We consider the p-divisor
$\D=\D_{\infty}\otimes \{\infty\}+\D_{0}\otimes \{0\}+\D_{1}\otimes \{1\}$ on $Y=\P^1$, where the piecewise-linear functions $\D_P:\omega\to\R$ induce subdivisions of $\Delta$, and has values on it vertices exactly as pictured in Figure \ref{fig:D}. For the  figure, we have projected $\Delta$ to the second and third coordinates.

Let $y\in\C(\P^1)$ be such that $\Div(y)=\{0\}-\{\infty\}$, that is, $y$ is a rational function vanishing only at the point $0$, with a single pole at $\infty$. We obtain that 
\[
	R(\D)=\C[y\chi^{(1,1,0)},\chi^{(1,0,0)},y\chi^{(1,0,0)},\chi^{(1,-1,0)},y \chi^{(1,0,1)}].
\]
Sending
\begin{align*}
	x_1\mapsto y \chi^{(1,1,0)},\  x_2\mapsto \chi^{(1,-1,0)},\ x_3\mapsto (1-y)\chi^{(1,0,0)},\ 
	x_4\mapsto y\chi^{(1,0,0)},\ x_5\mapsto y\chi^{(1,0,1)}
\end{align*}
induces an isomorphism of $R$ with $R(\D)$. This p-divisor $\D$ for $R$ could have been obtained with the method of \cite[\S11]{AH}.

Under this identification of $R$ with $R(\D)$, the valuation $\val_1$ is just the valuation $\val_Q$ for $Q=0\in\P^1$; likewise $\val_2=\val_Q$ for $Q=1\in\P^1$. Hence, to obtain $\Psi_i$ as in Theorem \ref{thm:NO}, we take $\Psi_0,\Psi_1,\Psi_2$ to respectively be the restrictions of $\D_\infty$, $\D_0$, and $\D_1$ to $\Delta$. See Figure \ref{fig:D}.
 One immediately checks that the description of $\Delta(A,\val_i)$ from Theorem \ref{thm:NO} holds.
 \subsection{Dual Description and Mutations}\label{sec:mutation}
 We continue with notation as in \S\ref{sec:wall}. In particular, $\Psi_0,\Psi_1,\Psi_2$ and $\Delta$ are as in Theorem \ref{thm:NO}.
Define polyhedra
\begin{align*}
	\nabla_i=\{u\in\R^{d}\ |\ \langle u,x\rangle \geq \Psi_i(x)\ \textrm{for all}\ x\in\Delta\}.
\end{align*}
Using Minkowski addition and the inclusion $\R^d\hookrightarrow \R^{d}\times\R=\R^{d+1}$ sending $u$ to $(u,0)$, we further define
\begin{align*}
	&\sigma_1=\cone\big(\nabla_1+e,\nabla_2+\nabla_0-e\big)\subset\R^{d+1}\\
	&\sigma_2=\cone\big(\nabla_2+e,\nabla_1+\nabla_0-e\big)\subset\R^{d+1}.
\end{align*}
Here $e$ denotes the $(d+1)$st standard basis vector of $\R^{d+1}$.
It is a straightforward exercise in convexity to see that we recover $\Delta(R,\val_i)$ by intersecting the dual cone $\sigma_i^\vee\subset \R^{d+1}$ with $\{1\}\times \R^d$.

Consider a vector $\eta$ in the interior of $\Delta$ with $\Psi_1(\eta),\Psi_2(\eta)> 0$ and $\Psi_0(\eta)=0$.
	We obtain polytopes 
\[
	D_i=\{u\in\sigma_i\ |\ \langle u,(\eta,0)\rangle =1 \}
\]
for $i=1,2$, from which we can recover $\sigma_i$, and thus $\Delta(R,\val_i)$.

\begin{remark}\label{rem:eta}
	In general, such $\eta$ may not exist. Nonetheless, for any $\eta$ in the interior of $\Delta$, there exist rational numbers $c_0,c_1,c_2$ with $c_0+c_1+c_2=0$ such that the $\Psi_i'=\Psi_i+c_i$ satisfy the desired condition. The polytopes resulting from the $\Psi_i'$ as described in Theorem \ref{thm:NO} will be rational translates of $\Delta(A,\val_1)$ and $\Delta(A,\val_2)$ in the direction of the final coordinate.
	
	On the other hand, the choice of $\eta$ is far from unique. However, if $\Delta(A,\val_1)$ (or $\Delta(A,\val_2)$) has a unique interior lattice point (for example, $\Delta(A,\val_1)$ is reflexive), a natural choice for $\eta$ is the projection to $M$ of this lattice point.
\end{remark}

The transition from $D_1$ to $D_2$ may be viewed as a generalization of the \emph{combinatorial mutations} considered in \cite{ACGK}, as we now explain.
To make this connection, we will assume that the point $(\eta,0)\in\Delta\times \R$ is contained in every facet of the graph of $\Psi_0$.
This is equivalent to requiring that $\langle u,\eta\rangle=0$ for each vertex $u$ of $\nabla_0$.

For $i=1,2$ we define 
\begin{align*}
	h_i&=\max_{u\in \nabla_i}\left(\frac{1}{\langle u,\eta \rangle}\right).
\end{align*}
Let $\lambda$ be the smallest natural number such that $\lambda/\langle u,\eta\rangle\in \N$ for all vertices $u$ of $\nabla_0,\nabla_1,\nabla_2$.
For any integer $\ell$, let $H_\ell=\{ u\in \R^d\ |\ \langle u,\eta \rangle =\ell/\lambda\}$. Set
$\tau=(\cone \Delta)^\vee\cap H_\lambda$ and $F=\nabla_0\cap H_0$.
Then we can rewrite $D_1$ as 
\begin{align*}
	D_1=
	\conv \Big(\tau, \bigcup_{\ell=1}^{\lambda\cdot h_1} \frac{\lambda}{\ell}\left[(\nabla_{1}\cap H_{\ell})+e\right],\quad\\ 
	\qquad\bigcup_{\ell=1}^{\lambda\cdot h_2} \frac{\lambda}{\ell}\left[(\nabla_{2}\cap H_{\ell})+F-e\right] \Big)
\end{align*}
and $D_2$ as 
\begin{align*}
	D_2=
	\conv \Big(\tau, \bigcup_{\ell=1}^{\lambda\cdot h_2} \frac{\lambda}{\ell}\left[(\nabla_{2}\cap H_{\ell})+e\right],\quad\\ 
	\qquad\bigcup_{\ell=1}^{\lambda\cdot h_1} \frac{\lambda}{\ell}\left[(\nabla_{1}\cap H_{\ell})+F-e\right] \Big).
\end{align*}
Comparing with \cite[Definition 5]{ACGK}, we see that up to mirroring the final coordinate, this is a combinatorial mutation, except that we have relaxed the integrality constraints from loc.~cit.

\subsection{Mutation Example}
To illustrate the connection to mutations, we present a second example.
We keep notation from \S \ref{sec:wall} and \S\ref{sec:mutation}.
Consider the ideal $I$ of $\C[x_1,\ldots,x_8]$ generated by 
\begin{align*}
      {x}_{4}+{x}_{5}-{x}_{6},\qquad  {x}_{3}-{x}_{6}-{x}_{8},\qquad      {x}_{2}-{x}_{5}+{x}_{6}-{x}_{8}\\
      {x}_{1}{x}_{7}-{x}_{5}{x}_{8},\qquad
      {x}_{5}{x}_{6}-{x}_{6}^{2}+{x}_{5}{x}_{8}\end{align*}
which is homogeneous with respect to the standard grading. Its tropicalization has three maximal cones, all of which are prime. These three cones intersect in $C$, the lineality space of $\trop(I)$. The respective initial ideals are
\begin{align*}
I_1=\langle
      {x}_{4}-{x}_{6},
      {x}_{3}-{x}_{8},
      {x}_{2}-{x}_{8},
      {x}_{1}{x}_{7}-{x}_{5}{x}_{8},
      {x}_{6}^{2}-{x}_{5}{x}_{8}\rangle\\
I_2=\langle
      {x}_{5}-{x}_{6},
      {x}_{4}-{x}_{8},
      {x}_{3}-{x}_{6},
      {x}_{1}{x}_{7}-{x}_{6}{x}_{8},
      {x}_{2}{x}_{6}-{x}_{8}^{2}\rangle\\
I_3=\langle
      {x}_{6}+{x}_{8},
      {x}_{4}+{x}_{5},
      {x}_{2}-{x}_{5},
      {x}_{1}{x}_{7}-{x}_{5}{x}_{8},
      {x}_{3}{x}_{5}-{x}_{8}^{2}\rangle.
      \end{align*}
All three are prime ideals; we focus on the ideals $I_1$ and $I_2$.
      We may take the elements $u_1,u_2,w_1,w_2$ to be the rows of 
the matrix
{\scriptsize{\[
\left(
\begin{array}{c c c c c c c c}
1&1&1&1&1&1&1&1\\
      {-1}&0&0&0&0&0&1&0\\
\hline
      0&{-1}&{-1}&0&1&0&0&{-1}\\
\hline
      0&1&{-1}&0&{-1}&{-1}&{-1}&0
\end{array}\right).
\]}}
\noindent Then $\Delta(A,\val_1)$ and $\Delta(A,\val_2)$ are the convex hulls of the columns of this matrix, after removing the fourth and third rows, respectively.

In this example, we may identify $M$ with $\Z^2$ via the first two rows of the above matrix. The polytope $\Delta$ is exactly the convex hull of $(1,1)$ and $(1,-1)$. The functions $\Psi_i$ from Theorem \ref{thm:NO} are as follows:
\begin{align*}
	\Psi_0(x,1)&=\begin{cases}
		0 & x\leq 0\\
		-x & x \geq 0
	\end{cases}\\
	\Psi_1(x,1)&=\begin{cases}
		x+1 & x\leq 0\\
		-x+1 & x \geq 0
	\end{cases}\\
	\Psi_2(x,1)&=\begin{cases}
		x+1 & x\leq 0\\
		1 & x \geq 0
	\end{cases}.
\end{align*}
This gives rise to 
\begin{align*}
	\nabla_0&=\conv\{(0,0),(0,-1)\}+\cone\{(1,1),(1,-1)\}\\
	\nabla_1&=\conv\{(1,1),(1,-1)\}+\cone\{(1,1),(1,-1)\}\\
	\nabla_2&=\conv\{(1,1),(1,0)\}+\cone\{(1,1),(1,-1)\}
\end{align*}
and  cones $\sigma_1$ and $\sigma_2$ generated respectively by the columns of the matrices
\begin{equation*}
	\tag{$\dagger\dagger$}
\label{eqn:mmm}
	\left(\begin{array}{c c c c}
1&1&1&1\\
1&-1&1&-1\\
1&1&-1&-1
	\end{array}
		\right)
		\qquad\textrm{and}\qquad\left(\begin{array}{c c c c}
1&1&1&1\\
1&0&1&-2\\
1&1&-1&-1
	\end{array}
		\right).
\end{equation*}

The natural choice of $\eta$ in this example (see Remark \ref{rem:eta}) is $\eta=(1,0)$. With this choice of $\eta$, we obtain that $D_1$ and $D_2$
are respectively the convex hulls of the columns of the matrices in \eqref{eqn:mmm}. We also obtain $h_1=h_2=\lambda=1$. Furthermore, 
\[
	\tau=\conv\{(1,-1),(1,1)\}\qquad F=\conv\{(0,0),(0,-1)\}.
\]
Considering Figure \ref{fig:ex}, we see that $D_1$ and $D_2$ are exactly as described at the end of \S\ref{sec:mutation}.
\begin{figure}

	\begin{tikzpicture}
		\draw[fill,color=lightgray] (1,1) -- (1,-1) -- (-1,-1) -- (-1,1);
\draw[fill] (1,1) circle [radius=.05];
\draw[fill] (0,1) circle [radius=.05];
\draw[fill] (-1,1) circle [radius=.05];
\draw[fill] (1,0) circle [radius=.05];
\draw[fill] (0,0) circle [radius=.05];
\draw[fill] (-1,0) circle [radius=.05];
\draw[fill] (1,-1) circle [radius=.05];
\draw[fill] (0,-1) circle [radius=.05];
\draw[fill] (-1,-1) circle [radius=.05];
\draw[dashed,very thick] (-1,0) -- (1,0);
\draw[dashed,very thick] (-1,1) -- (1,1);
\draw[dashed,very thick] (-1,-1) -- (1,-1);
\node [above] at (0,1.1) {\scriptsize{$\nabla_1\cap H_1+e$}};
\node [above] at (.2,0) {\scriptsize{$\tau$}};
\node [below] at (0,-1.1) {\scriptsize{$\nabla_2\cap H_1+F-e$}};
\node [right] at (1,1) {\scriptsize{$(1,1,1)$}};
\node [right] at (1,-1) {\scriptsize{$(1,1,-1)$}};
\node [left] at (-1,1) {\scriptsize{$(1,-1,1)$}};
\node [left] at (-1,-1) {\scriptsize{$(1,-1,-1)$}};
\node  at (0,-2) {$D_1$};
	\end{tikzpicture}
	\hspace{1cm}	\begin{tikzpicture}
		\draw[fill,color=lightgray] (1,1) -- (1,-1) -- (-2,-1) -- (0,1);
\draw[fill] (1,1) circle [radius=.05];
\draw[fill] (0,1) circle [radius=.05];
\draw[fill] (-2,-1) circle [radius=.05];
\draw[fill] (1,0) circle [radius=.05];
\draw[fill] (0,0) circle [radius=.05];
\draw[fill] (-1,0) circle [radius=.05];
\draw[fill] (1,-1) circle [radius=.05];
\draw[fill] (0,-1) circle [radius=.05];
\draw[fill] (-1,-1) circle [radius=.05];
\draw[dashed,very thick] (-1,0) -- (1,0);
\draw[dashed,very thick] (0,1) -- (1,1);
\draw[dashed,very thick] (-2,-1) -- (1,-1);
\node [above] at (.5,1.1) {\scriptsize{$\nabla_2\cap H_1+e$}};
\node [above] at (.2,0) {\scriptsize{$\tau$}};
\node [below] at (-.5,-1.1) {\scriptsize{$\nabla_1\cap H_1+F-e$}};
\node [right] at (1,1) {\scriptsize{$(1,1,1)$}};
\node [right] at (1,-1) {\scriptsize{$(1,1,-1)$}};
\node [left] at (0,1) {\scriptsize{$(1,0,1)$}};
\node [left] at (-2,-1) {\scriptsize{$(1,-2,-1)$}};
\node  at (0,-2) {$D_2$};
	\end{tikzpicture}

	\caption{A combinatorial mutation}\label{fig:ex}

\end{figure}
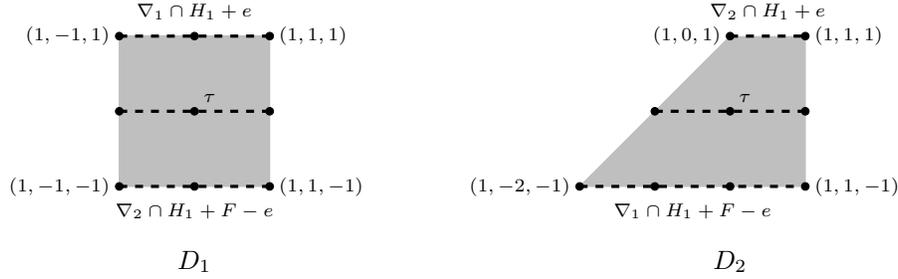

\section*{Acknowledgements}
Both authors are grateful for the support, hospitality, and the research environment of the Fields Institute for Research in Mathematics in Toronto, Canada, where this project started. We also thank the Association for Women in Mathematics for the Mentoring Travel Grant awarded to the first author, which enabled us to work on this project. In addition, we thank the Osaka City University Advanced Mathematical Institute, where parts of this paper were written. We are indebted to Nathan Ilten for patiently explaining the complexity-one perspective to us and for writing the Appendix on such short notice. We are equally indebted to Diane Maclagan for simplifying our proof of Theorem~\ref{theorem: fiber lengths equal}. We also thank Chris Manon, Kiumars Kaveh, and Kristin Shaw for crucial conversations and suggestions. We are grateful to Adam Van Tuyl for helping us find Example~\ref{example: algebraic is not geometric} and to Anders Jensen for answering questions about tropical Grassmannians. The first author is grateful to Jason Anema, Federico Ardila, and Patricio Gallardo for helpful conversations.
We would like to thank the referee for carefully reading the manuscript and providing insightful comments that improved the exposition.

\end{document}